\documentclass[12pt,letterpaper,reqno]{amsart}

\usepackage{bm}
\usepackage{mathrsfs}
\usepackage{latexsym}
\usepackage{epsfig}
\usepackage{color}
\usepackage{float}
\usepackage{amsmath,amsfonts,amsthm,amssymb,amscd,mathtools,hyperref}
\usepackage[all,dvips,arc,curve,frame]{xy}

\newcommand{\supp}{\operatorname{supp}}

\newcommand{\Z}{\mathbb{Z}}

\newcommand{\C}{\text{Com}(\Gamma)}
\newcommand{\X}{\Gamma\backslash G}
\newcommand{\Q}{\mathbb{Q}}
\newcommand{\R}{\mathbb{R}}
\newcommand{\stab}{\text{stab}}

\newcommand{\y}{\widetilde{y}}
\newcommand{\z}{\widetilde{z}}

\begin{document}

\newtheorem{thm}{Theorem}[section]
\newtheorem{conj}[thm]{Conjecture}
\newtheorem{cor}[thm]{Corollary}
\newtheorem{lem}[thm]{Lemma}
\newtheorem{prop}[thm]{Proposition}
\newtheorem{exa}[thm]{Example}
\newtheorem{defi}[thm]{Definition}
\newtheorem{rem}[thm]{Remark}

\numberwithin{equation}{section}

\title[M\"{o}bius disjointness for homogeneous dynamics]{M\"{o}bius disjointness for homogeneous dynamics} 
\author{Ryan Peckner}
\address{Mathematics Department, Princeton University, Princeton, NJ 08544}
\email{rpeckner@princeton.edu}

\maketitle

\begin{abstract}
We prove Sarnak's M\"{o}bius disjointness conjecture for all unipotent translations on homogeneous spaces of real connected Lie groups. Namely, we show that if $G$ is any such group, $\Gamma\subset G$ a lattice, and $u\in G$ an Ad-unipotent element, then for every $x\in\X$ and every continuous, bounded function $f$ on $\X$, the sequence $f(xu^{n})$ cannot correlate with the M\"{o}bius function on average.
\end{abstract}

\section{Introduction}

One of the most important objects in analytic number theory is the M\"obius function, defined for positive integers $n$ by
\[ \mu(n) = \left\{
\begin{array}{cl}
1 &\ \text{if }n=1 \\
0 &\ \text{if }n\text{ is not squarefree}\\
(-1)^{r} &\ \text{if }n=p_{1}\cdots p_{r}\text{ is a product of $r$ distinct primes}.
\end{array}\right. 
\]
The overall behavior of the M\"{o}bius function is captured quantitatively by the summatory function
\[
S(N) := \sum_{n\leq N}\mu(n)
\]
as $N\to\infty$. Sign fluctuations in $\mu(n)$, resulting from randomness in the distribution of prime numbers, create cancellations in $S(N)$ that reflect specific number-theoretic phenomena. For instance, the prime number theorem is elementarily equivalent to the estimate $S(N) = o(N)$, while the Riemann hypothesis is equivalent to the power saving $S(N) = O_{\epsilon}(N^{1/2 + \epsilon})$ for all $\epsilon > 0$. 

Sarnak has formulated a less quantitative description of M\"{o}bius randomness from this standpoint that seeks to measure correlations of $\mu(n)$ with simpler functions, rather than its  intrinsic cancellations \cite{S}. This simultaneously gives a rigorous grounding to the well-known ``M\"{o}bius randomness heuristic" \cite{IK}, which states that the sum
\[
\sum_{n\leq N}\mu(n)a_{n}
\]
should be asymptotically small whenever $a_{n}$ is a ``reasonable" sequence of complex numbers, and generalizes more precise results along these lines. These include Vinogradov's estimate for exponential sums
\[
\sum_{n\leq N}\mu(n)\exp(2\pi i\alpha n) \ll_{A} \frac{N}{\log^{A}N}\text{ for all }A>0,
\]
and the generalization thereof to nilmanifolds due to Green-Tao \cite{GT}. 

Sarnak's conjecture on the randomness of the M\"{o}bius function is set in the context of deterministic topological dynamical systems. It asserts that if $X$ is a compact metric space and $T: X\to X$ a continuous surjection such that the topological entropy of $(X,T)$ is zero, then for every $x\in X$ and every continuous function $f\in C(X)$, we have
\[
\frac{1}{N}\sum_{n=0}^{N-1}\mu(n)f(T^{n}x)\to 0\text{ as }N\to\infty.
\]
Thus, the notion of a reasonable sequence of complex numbers in this setting is that of an observable sequence from a topological dynamical system of zero entropy.

This statement deliberately avoids demanding a rate of convergence, as these are currently available only for relatively simple systems. The fundamental nuance in this formulation is that $x$ may be \emph{any} point whatsoever in the space $X$. If we relax the convergence to merely hold almost everywhere with respect to some $T$-invariant measure on $X$, the result follows from the spectral theory of the Koopman operator for the system and Bourgain's version of the ergodic theorem \cite{B2}. 

This conjecture has been verified for certain rank-one transformations \cite{B1, ALR}, skew-products on tori \cite{LS}, Kronecker systems and rotations on nilmanifolds \cite{Dav, GT}, and horocycle flows on surfaces of constant negative curvature \cite{BSZ}, among other cases. The latter two examples, although quite different in nature, are conceptually of the same ilk. Both are homogeneous dynamical systems of zero entropy, meaning that there is a Lie group $G$ (a nilpotent group for the nilmanifold case, and $SL_{2}(\mathbb{R})$ for horocycle flows) and a lattice $\Gamma\subset G$ such that the dynamics of the system are given by a zero-entropy homeomorphism of the homogeneous space $\Gamma\backslash G$. 

The most basic class of such zero-entropy homeomorphisms is obtained by allowing an Ad-unipotent element $u\in G$ to act on $\Gamma\backslash G$ by right translation. This yields a wide range of dynamical behavior, from a distal system when $G$ is nilpotent to one that is much more complex when $G$ is semisimple (such as for horocycle flows), being mixing of all orders \cite{Sta}. The profound results of Ratner on the topological and measure rigidity of such unipotent translations \cite{R1, R2, R3, R4} render them a fertile testing ground for many phenomena in dynamics, and indeed her results are essential for proving the disjointness of $\mu(n)$ from horocycle flows \cite{BSZ}.

In this paper, we use the full generality of Ratner's theorems to prove the disjointness of the M\"{o}bius function from all unipotent translations on homogeneous spaces of connected Lie groups.
\begin{thm}
Let $G$ be a real connected Lie group, $\Gamma\subset G$ a lattice, and $u\in G$ an Ad-unipotent element. Then for every $x\in \Gamma\backslash G$ and every continuous, bounded function $f$ on $\Gamma\backslash G$, we have
\[
\frac{1}{N}\sum_{n\leq N}\mu(n)f(xu^{n})\to 0 \text{ as }N\to\infty.
\]
\label{thm: main}
\end{thm}

This does not quite settle Sarnak's conjecture for all zero-entropy homogeneous dynamics, as we do not account for the more general class of quasi-unipotent affine automorphisms. These are maps on $\X$ of the form $T_{g}\circ\alpha$, where $\alpha$ is an automorphism of $G$ such that $\alpha(\Gamma) = \Gamma$, and $g\in G$ has the property that all eigenvalues of $\text{Ad}\,g\circ \,\textrm{d}\alpha$ are of absolute value 1 ($T_{g}$ denotes right translation by $g$ on $\X$). However, we strongly suspect that the ideas we present in proving the disjointness conjecture for unipotent translations should generalize in a fairly straightforward way to these systems. 

It should be noted also that while we only prove our result for real connected Lie groups, the statement and method of proof should be roughly similar for the case of $p$-adic Lie groups and their products, as Ratner rigidity continues to hold in this context \cite{R4}. 

One may view this result as the statement that the M\"{o}bius function is \emph{linearly} disjoint from all deterministic homogeneous dynamical systems of the stated kind. Extending this statement beyond a linear expression in $\mu(n)$ seems out of reach at present; even to produce significant cancellations in the second-order correlations sum
\[
\sum_{n\leq N}\mu(n)\mu(n+1)
\] 
presents an enormous challenge, and may be as difficult as the twin prime conjecture. The best known bound for this sum is the recent result of Matom\"{a}ki and Radziwill \cite{MR}: there exists some $\delta > 0$ such that 
\[
\frac{1}{N}\left|\sum_{n\leq N}\lambda(n)\lambda(n+1)\right| \leq 1 - \delta
\]
for all sufficiently large $N$, where $\lambda$ is the Liouville function $\lambda(n) = (-1)^{\Omega(n)}$, with $\Omega(n)$ the number of prime divisors of $n$ counted with multiplicity. See also \cite{C} for related results. 

In this vein, Sarnak has shown that the M\"{o}bius disjointness conjecture follows from the well-known Chowla conjecture \cite{T}, which states that if $m> 0$ and $a_{1},a_{2},\dots,a_{m}$ are nonnegative integers, at least one of which is odd, then
\[
\sum_{n\leq N}\mu(n+1)^{a_{1}}\mu(n+2)^{a_{2}}\cdots\mu(n+m)^{a_{m}} = o(N).
\]
An enriched version of the disjointness conjecture, allowing for correlations between such nonlinear expressions in $\mu$ and corresponding expressions in a deterministic sequence, is therefore unapproachable until the Chowla conjecture is known. 

It is our hope that our result may apply to the study of solutions of polynomial equations in prime numbers, as in \cite{GT}. However, such applications would require a stronger version of our theorem, in which the rate of vanishing of the correlated sum is quantified. As mentioned above, no such rates are known in the general setting of Ratner's theorems, though partial results have been obtained in this direction \cite{SU}. 

\subsection{Necessary background} We begin by recording several powerful results as the fundamental input for our work. First, the following criterion allows us to establish M\"{o}bius disjointness by studying correlations of powers of the $u$-action. 

\begin{thm}[The disjointness criterion, \cite{BSZ} Thm. 2]
 Let $F: \mathbb{N}\to\mathbb{C}$ with $|F|\leq 1$ and let $\nu$ be a multiplicative function with $|\nu|\leq 1$. There exists a $\tau_{0} > 0$ such that the following holds: let $\tau \leq \tau_{0}$ and assume that for all primes $p,q\leq e^{1/\tau}, p\neq q$, we have that for $M$ large enough
\begin{equation}
 \left|\sum_{m\leq M}F(pm)\overline{F(qm)}\right|\leq\tau M.
 \label{eq: small}
\end{equation}
Then for $N$ large enough,
\begin{equation}
 \left|\sum_{n\leq N}\nu(n)F(n)\right|\leq 2\sqrt{\tau\log(1/\tau)}N.
\end{equation}
\label{thm: criterion}
\end{thm}
In fact, the conclusion of this statement can be drawn even when the inequality (\ref{eq: small}) is violated on a small set of primes, whose size may increase as $\tau\to 0$. We will have more to say on this later.

Thus, our aim is to analyze the correlation limits
\begin{equation}
 \lim_{N\to\infty}\frac{1}{N}\sum_{n=1}^{N}f(xu^{\lambda n})f(xu^{\mu n}) 
\label{eq: corr}
\end{equation}
where, in the notation of Theorem \ref{thm: main}, $x\in X = \X, f\in C_{b}(X)$, and $\lambda, \mu\in\mathbb{N}$. Ratner's uniform distribution theorem allows us to describe such limits in terms of ergodic measures arising naturally from subgroups of $G$. If $x\in X$, we say that the closure $\overline{\{xu^{k} : k\in\Z\}}$ of the $u$-orbit of $x$ is homogeneous if there is a closed subgroup $H\subset G$ such that $u\in H, \widetilde{x}H\widetilde{x}^{-1}\cap \Gamma$ is a lattice in $\widetilde{x}H\widetilde{x}^{-1}$, where $\Gamma\widetilde{x} = x$, and $\overline{\{xu^{k} : k\in\Z\}} = xH$. 

\begin{thm}[Uniform distribution of unipotent trajectories - \cite{R2} Theorem B and Corollary A (1)]
 Let $G$ be a connected Lie group, $\Gamma\subset G$ a lattice, and $u\in G$ an Ad-unipotent element. Then every orbit of $u$ acting on $\X$ is homogeneous and uniformly distributed in its closure. More precisely, for every $x\in X$, there is a closed subgroup $H\subset G$ such that $\overline{\{xu^{k} : k\in\Z\}}$ is homogeneous with respect to $H$ as above, and we have
\[
 \lim_{N\to\infty}\frac{1}{N}\sum_{n=0}^{N-1}f(xu^{n}) = \int_{X}f\,\mathrm{d}\nu_{H}
\]
for any bounded continuous function $f$ on $X$, where $\nu_{H}$ is an $H$-invariant Borel probability measure supported on $xH$.

Moreover, if $H$ is the smallest closed subgroup of $G$ with respect to which the closure of the $u$-orbit of $x$ is homogeneous, then the action of $u$ on $(xH,\nu_{H})$ is ergodic.
\label{thm: equidistribution}
\end{thm}

A probability measure of the form appearing in the above theorem is called \emph{algebraic}. Applying this theorem to the action of the unipotent element $(u^{\lambda},u^{\mu})$ in the group $G\times G$, we have that the limit (\ref{eq: corr}) exists for any $\lambda,\mu> 0$, and is given by 
\begin{equation}
  \lim_{N\to\infty}\frac{1}{N}\sum_{n=1}^{N}f(xu^{\lambda n})f(xu^{\mu n}) = \int_{\X\times\X}(f\otimes f)(x,y)\,\textrm{d}\nu(x,y) 
\label{eq: limit}
\end{equation}
where $\nu$ is an algebraic ergodic probability measure on the system $(X\times X,u^{\lambda}\times u^{\mu})$. 


Our proof proceeds by considering the various possibilities for the subgroup $H$ appearing in Theorem \ref{thm: equidistribution}, relative to which the $u$-orbit of the point $x$ is homogeneous. In section \ref{sec: semisimple}, we consider the case in which $H$ is semisimple, where we use geometric arguments to show that the correlation limit measure appearing in (\ref{eq: limit}) must be the trivial joining of the Haar measure of $H$ with itself, for all but finitely many primes $p\neq q$. In section \ref{sec: levi}, we consider the case that $H$ is any connected Lie group satisfying the conditions of Theorem \ref{thm: equidistribution}, and use the Levi decomposition to combine our work in the semisimple case with the known case of M\"{o}bius disjointness for nilmanifolds. 

We begin by outlining our method in the case $G = SL_{2}(\mathbb{R})$. While this case has already been addressed in \cite{BSZ}, we provide this concrete example in order to illustrate our general approach in the case of semisimple groups. 

\section{The case $G = SL_{2}(\mathbb{R})$}\label{sec: sl2} 

Suppose that $\Gamma\subset SL_{2}(\mathbb{R})$ is an arithmetic lattice, and let $x\in \X$. By a theorem of Dani, the closure of the $u$-orbit of $x$ is one of a finite set, a circle, or all of $\X$. In the first two cases, the disjointness of $\mu(n)$ from the $u$ orbit of $x$ is well-known to be true (by Dirichlet's theorem on primes in arithmetic progressions in the first case, and by a theorem of Vinogradov in the second \cite{Dav}).

Thus, we suppose that the $u$ orbit of $x$ is uniformly distributed with respect to the Haar measure $\nu_{G}$. Then we have by Theorem \ref{thm: equidistribution} 
\[
\lim_{N\to\infty}\frac{1}{N}\sum_{n\leq N}f(xu^{pn})\overline{f(xu^{qn})} = \int_{\X\times\X}f\otimes\overline{f}\,\textrm{d}\lambda,
\]
where $\lambda$ is an ergodic joining of $(\X,u^{p},\nu_{G})$ with $(\X,u^{q},\nu_{G})$. We are clearly finished if $\lambda = \nu_{G}\times\nu_{G}$, since by the prime number theorem, the statement of M\"{o}bius disjointness is unaffected if we subtract $\int_{\X}f\,\textrm{d}\nu_{G}$ from $f$.

Thus, suppose that there exists a nontrivial joining of $(\X,u^{p},\nu_{G})$ with $(\X,u^{q},\nu_{G})$. By Ratner's joinings theorem (\cite{R1}, or see Theorem \ref{thm: joinings} below), there exists some $\beta\in\C$ such that
\[
\beta u\beta^{-1} = u^{p/q},
\]
and
\[
\widetilde{x}\beta\widetilde{x}^{-1}\in\C,
\]
where $\Gamma\widetilde{x} = x$. It follows that $\widetilde{x}\beta\widetilde{x}^{-1}$ belongs to $P\cap\C$, where $P$ is the parabolic subgroup $N_{G}(\widetilde{x}U\widetilde{x}^{-1})$ ($U$ being the one-parameter unipotent subgroup generated by $u$). The fact that the $u$-orbit of $x$ equidistributes with respect to $\nu_{G}$ implies that, on realizing $P$ as the stabilizer $\stab_{G}(z)$ of a point $z\in\mathbb{P}^{1}(\R)$, we have $z\not\in\mathbb{P}^{1}(\Q)$. Hence, we will be finished if we can show the following.

\begin{prop}
 Let $P\subset G$ be a proper parabolic subgroup, so that $P = P_{z}$ is the stabilizer of some $z\in\mathbb{P}^{1}(\R)$, and suppose that $z\not\in\mathbb{P}^{1}(\Q)$. Let $\chi: P\cap\C\to\mathbb{R}^{*}_{>0}$ be a real character. Then $\chi(P\cap\C)\cap\Q^{*} = \{1\}$.
\end{prop}
Letting $\chi$ be the real character of $N_{G}(U)$ defined by
\[
\gamma u \gamma^{-1} = u^{\chi(\gamma)},
\]
which then determines a character of any conjugate of $P$, it follows that in the setting above, we must in fact have $p=q$, i.e. there are no nontrivial joinings of $(\X,u^{p},\nu_{G})$ with $(\X,u^{q},\nu_{G})$ if $p\neq q$. Theorem \ref{thm: criterion} then implies Theorem \ref{thm: main}.

We divide into cases according to whether $\Gamma$ is a uniform lattice in $G$. 

\subsection{$\Gamma\backslash G$ is compact} In this case, $\Gamma$ is commensurable with a unit group in a quaternion division algebra $D$ defined over a totally real number field $K$ \cite{We}. Let $D$ be generated linearly over $K$ by $1,\omega, \Omega$, and $\omega\Omega$, where $\omega^{2}=a, \Omega^{2} =b$ with $a,b\in K$ and $a>0$ (this makes sense since $K$ is totally real). 

Consider the $4\times 4$ matrix representation
\[
 \psi: D\to M_{4}(\mathbb{R}) ,
 \]
obtained by mapping $\alpha = x_{0} + x_{1}\omega + x_{2}\Omega + x_{3}\omega\Omega$ to the matrix
\[
x_{0}\begin{pmatrix} 1 & 0 & 0 & 0 \\ 0 & 1 & 0 & 0 \\ 0 & 0 & 1 & 0 \\ 0 & 0 & 0 & 1\end{pmatrix} + x_{1}\begin{pmatrix} 0 & 1 & 0 & 0 \\ a & 0 & 0 & 0 \\ 0 & 0 & 0 & 1 \\ 0 & 0 & a & 0 \end{pmatrix} + x_{2}\begin{pmatrix} 0 & 0 & 1 & 0 \\ 0 & 0 & 0 & -1 \\ b & 0 & 0 & 0 \\ 0 & -b & 0 & 0 \end{pmatrix}
+ x_{3}\begin{pmatrix} 0 & 0 & 0 & 1 \\ 0 & 0 & -a & 0 \\ 0 & b & 0 & 0 \\ -ab & 0 & 0 & 0\end{pmatrix}.
\]
Let $D_{1}$ be the group of unit quaternions; this is identified with the group of $\mathbb{R}$ points of an algebraic $K$-subgroup of $SL_{4}$ under this representation. The group of rational points     $G:= \text{Res}_{K/\Q}(D_{1})$ is then a subgroup of $SL_{m}(\mathbb{R})$ for some $m\geq 4$. $\Gamma$ is isomorphic over $K$ to a lattice commensurable with $G_{\Z}$, and so we have a $K$-isomorphism
\[
 \C\cong \text{Com}(G_{\Z}). 
\]
Up to scalar multiples of the identity, $\text{Com}(G_{\Z})$ is simply $G_{\Q}$. 

Now, suppose that $P$ is a parabolic subgroup of $SL_{2}(\R)$ defined over $\R$, and identify this with a parabolic subgroup of $G$ (we will still use $P$ to denote this group). Let $H$ be the Zariski closure of $P\cap G_{\Q}$ inside $SL_{m}(\R)$ (i.e. the smallest zero locus of a set of polynomials with real coefficients containing $P\cap G_{\Q}$). Observe that $H$ is defined over $\Q$, since it consists of rational matrices. Let $\chi: H\to\R^{*}_{>0}$ be a character. If $P$ is conjugate to the group of unipotent upper triangular matrices in $SL_{2}(\R)$, then $\chi$ must be trivial, so we assume instead that $P$ is conjugate to the group of all upper triangular matrices. Consequently, if $g\in P$ is conjugate to an upper-triangular matrix with diagonal entries $\lambda, \mu = \lambda^{-1}$, then $\chi(g) = \chi(\text{diag}(\lambda,\mu))$.   

Since $H\subset P$, it follows from this that $\chi$ is determined by its behavior on a maximal $\Q$-algebraic torus $S$ in $H$. Since $\Gamma$ is a uniform lattice in $SL_{2}(\R)$, so too is $G_{\Z}$ in $G_{\R}$, and therefore $G$ is $\Q$-anisotropic by the theorem of Borel and Harish-Chandra. This implies that $G$ contains no nontrivial $\Q$-split tori, so this is of course true of $H$ as well. 

%

Therefore, $S$ must be $\Q$-anisotropic, so the restriction $\chi|_{S}$ is defined over a number field $F$ with $d:= [F:\Q]>1$. Since $S_{\Q}$ consists of rational matrices, we have $\chi(S_{\Q})\subset F^{*}$. Thus, if we let $N_{F/\Q}:F^{*}\to\Q^{*}$ be the norm map, the composite $N_{F/\Q}\circ\chi:S_{\Q}\to\Q^{*}$ is well-defined. However, this composite is a $\Q$-character of $S_{\Q}$ (since it's fixed by the action $\chi^{\sigma}(g) = \sigma(\chi(\sigma^{-1}(g)))$ of $\text{Gal}(F/\Q)$), and therefore must be trivial (as otherwise it could be extended to a nontrivial $\Q$-character of the $\Q$-anisotropic torus $S$).

It follows that if $\chi(s)\in\Q^{*}$ for some $s\in S_{\Q}$, then
\[
 N_{F/\Q}(\chi(s)) = 1 = \chi(s)^{d}
\]
and since $d > 1$, this implies $\chi(s) = 1$. Therefore, $\chi(S_{\Q})\cap\Q^{*} = \{1\}$, and from the above it follows that $\chi(P\cap G_{\Q})\cap\Q^{*} = \{1\}$ as well.

These arguments immediately extend from $P\cap G_{\Q}$ to $P\cap\C$: if $B\in\C$, then $B^{n}$ belongs to $D_{\Q}$ for some $n\geq1$. Therefore, we get $\chi(B^{n}) = 1$ if $B^{n}$ belongs to a subgroup of $H$ on which $\chi$ is trivial, and it follows in this case that $\chi(B) = 1$ since $\chi$ is a character.

\subsection{$\Gamma\backslash G$ is noncompact} In this case, $\Gamma$ is commensurable with $SL_{2}(\mathbb{Z})$, and thus its commensurator is
\[
 \C = \left\{\frac{1}{(\det A)^{1/2}}A : A\in GL_{2}(\mathbb{Q})\right\}.
\]
Suppose $P\subset G$ is a parabolic subgroup stabilizing a point $z$ in $\mathbb{P}^{1}(\mathbb{R})\backslash \mathbb{P}^{1}(\Q)$. Let $\chi: P\cap\C\to\mathbb{R}^{*}$ be a real character. We first restrict our attention to $P\cap SL_{2}(\mathbb{Q})$. This consists of those rational matrices that stabilize $z$; in particular, this group is trivial if $z$ is not quadratic, so we assume there exist $a,b,c\in\Z$ such that $az^{2} + bz + c = 0$, with $(a,b,c)=1$ and $d=b^{2}-4ac > 0$. Then we have
\[
 P\cap\C = \left\{\begin{pmatrix} \frac{t+bu}{2} & au \\ -cu & \frac{t-bu}{2} \end{pmatrix} : t^{2}-du^{2}\in\Q^{+}, t,u\in\Q\right\}.
\]
The Zariski closure of $P\cap SL_{2}(\Q)$ is defined over $\Q$. Thus, an argument similar to that in the cocompact case implies that it suffices to show that $P$ does not contain a $\Q$-split torus. Suppose to the contrary that it does contain a $\Q$-split torus $S$. This torus is conjugate via an element of $GL_{2}(\Q)$ to the diagonal subgroup of $SL_{2}(\R)$, so $S$ stabilizes a point in $\mathbb{P}^{1}(\Q)$. The unipotent radical $U$ of $P$ contracts $S$, i.e. we have 
\[
 a^{-t}ua^{t}\to e\text{ as }t\to\infty\text{ for all }u\in U, \{a^{t}\}\subset S.
\]
Thus if $\{a^{t}\}\subset S$ is a one-parameter subgroup such that the geodesic ray $\{a^{t}y\}_{t=0}^{\infty}$ converges to $z$ for some $y\in\mathbb{H}$, then we have for any $u\in U$
\[
 \lim_{t\to\infty}d_{\mathbb{H}}(ua^{t}y, a^{t}y) = \lim_{t\to\infty}d(a^{-t}ua^{t}y, y) <\infty
\]
 since the set $\{a^{-t}ua^{t}:t\geq 0\}$ is bounded. It follows that $U$ and therefore all of $P$ fixes the same boundary point as $S$. But this contradicts the hypothesis that $P$ stabilizes $z\in\mathbb{P}^{1}(\mathbb{R})\backslash\mathbb{P}^{1}(\Q)$.

\section{The general semisimple case}\label{sec: semisimple} Let $G$ be the real locus of a connected, semisimple algebraic group of noncompact type with finite center defined over $\mathbb{Q}$ with $G\subset SL_{n}(\R)$ for some $n$, $\Gamma$ a lattice in $G$, and $u$ an Ad-unipotent element of $G$, acting by right translation on the homogeneous space $X = \X$. We aim to prove the following.

\begin{thm}
For any point $x\in X$ whose $u$-orbit is equidistributed relative to the $G$-invariant Borel probability measure $\nu_{G}$ on $X$ and any continuous, bounded function $f$ on $X$, we have
 \begin{equation}
 \lim_{N\to\infty}\frac{1}{N}\sum_{n=0}^{N-1}\mu(n)f(xu^{n}) = 0.
 \end{equation}
 \label{thm: disjoint}
\end{thm}

Let $u=\exp v$ where $v\in\frak{g}$ is nilpotent, and define the one-parameter unipotent subgroup
\[
 U = \{\exp(tv): t\in\R\} = \left\{u^{t}\right\}.
\]
Theorem \ref{thm: criterion} tells us that in order to establish M\"{o}bius disjointness, we must study correlations of the $u^{p}$ and $u^{q}$ actions, where $p\neq q$ are distinct primes. Applying Ratner's uniform distribution theorem (Theorem \ref{thm: equidistribution}) to the action of $(u^{p}, u^{q})$ on $\X\times\X$, we have
\begin{equation}
\lim_{N\to\infty}\frac{1}{N}\sum_{n\leq N}f(xu^{pn})\overline{f(xu^{qn})} = \int_{\X\times\X}f\otimes\overline{f}\,\textrm{d}\nu,
\label{eq: semisimple correlations}
\end{equation}
where $\nu$ is an ergodic algebraic joining of $(\X,u^{p},\nu_{G})$ with $(\X,u^{q},\nu_{G})$ which is supported on the closed orbit $\overline{\{(xu^{pn},xu^{qn})\}}_{n\in\mathbb{N}}$. 

Write $x = \Gamma\xi$, where $\xi\in G$. Let $N^{u} = N_{G}(U)$ and $N^{u}_{\xi} = \xi N^{u}\xi^{-1}$. For the time being, we restrict our attention to the case when $\Gamma$ is an \emph{irreducible} lattice in $G$. Then by Ratner's joinings theorem (\cite{R1}, though we use the implication in the irreducible case from Theorem 3.8.2 in \cite{KSS}), if $\nu \neq \nu_{G}\times\nu_{G}$, there exists $\beta\in G$ with
\[
 \beta u\beta^{-1} = u^{p/q} 
\] 
(so $\beta\in N^{u}$) such that 
\[
\xi\beta\xi^{-1}\in \C \cap N^{u}_{\xi}.
\]

There is a character $\chi: N^{u}\to\R_{>0}^{*}$ determined by
\[
 \alpha u \alpha^{-1} = u^{\chi(\alpha)} \text{ for } \alpha\in N^{u} = N_{G}(U).
\]
This character is then also well-defined on any conjugate of $P$, and we see from the above that
\[
 \chi\left(\xi\beta\xi^{-1}\right) = p/q.
\]
In order to apply Theorem \ref{thm: criterion}, we will show that the existence of such a $\beta$ forces the pair $(p,q)$ to belong to a finite set determined by $x$. 
\begin{thm}
For the character $\chi$ described above, the intersection
 \[
  \chi\left(\C\cap N^{u}_{\xi}\right)\cap\left\{\frac{p}{q}: p,q \text{ are prime}\right\} 
 \]
 is finite. Therefore, if $x\in\X$ is $u$-generic for the Haar measure $\nu_{G}$, then there are at most finitely many pairs of primes $p\neq q$ such that there is a nontrivial joining of $(\X,u^{p},\nu_{G})$ with $(\X,u^{q},\nu_{G})$ which is supported on the closed orbit $\overline{\{(xu^{pn},xu^{qn})\}}_{n\in\mathbb{N}}$.
 \label{thm: character}
\end{thm}

The proof of this is divided according to the nature of the irreducible lattice $\Gamma$. Recall that a lattice $\Gamma$ is called \emph{arithmetic} if there exist
\begin{enumerate}
 \item A closed, connected, semisimple subgroup $G'$ of some $SL_{m}(\R)$ such that $G'$ is defined over $\Q$, \\
\item Compact, normal subgroups $K$ and $K'$ of $G$ and $G'$, respectively, and \\
\item an isomorphism $\phi: G/K \to G'/K'$,
\end{enumerate}
such that $\phi(\overline{\Gamma})$ is commensurable to $\overline{G'_{\Z}}$, where $\overline{\Gamma}$ and $\overline{G'_{\Z}}$ are the images of $\Gamma$ and $G'_{\Z}$ in $G/K$ and $G'/K'$, respectively (\cite{WM2} Def. 5.18).

For nonuniform lattices in the kinds of groups on which we'll focus, there is no need for the compact subgroups arising in this definition:
\begin{prop}[\cite{WM2} Cor. 5.27]
 Assume $\Gamma$ is an irreducible, arithmetic, nonuniform lattice in the semisimple group $G$, which is connected and has no compact factors. Then, possibly after replacing $G$ by an isogenous group, there is an embedding of $G$ in some $SL_{m}(\R)$ such that $G$ is defined over $\Q$ and $\Gamma$ is commensurable to $G_{\Z}$.
\label{prop: commensurable}
\end{prop}

First, we may easily prove Theorem \ref{thm: character} in the case that $\Gamma$ is a non arithmetic lattice.

\begin{proof}[Proof of Thm \ref{thm: character} when $\Gamma$ is not an arithmetic lattice in $G$]
We have the following theorem of Margulis:
\begin{thm}[\cite{Mar}] 
Suppose $\Gamma$ is an irreducible lattice in the semisimple Lie group $G$. Then $\Gamma$ is arithmetic if and only if $\C$ is dense in $G$. Moreover, if $\Gamma$ is nonarithmetic, then $\C$ is a lattice in $G$ as well. 
\label{thm: margulis}
\end{thm}
Hence, we have in this case that $\C$ is a lattice in $G$. In particular, $\C\cap N^{u}_{\xi}$ is finitely generated as an abstract group, and therefore so too is $\chi(\C\cap N^{u}_{\xi})$. It follows that the set $\{p/q\in\chi(\C\cap N^{u}_{\xi}) : p,q \text{ are prime}\}$ is finite.  
\end{proof}

We now turn to proving Thm. \ref{thm: character} when $\Gamma$ is arithmetic. As we now show, this reduces to a statement about $\Q$-split tori inside $\C\cap N^{u}_{\xi}$. Let $H$ be the identity component of the Zariski closure in $SL_{n}(\mathbb{R})$ of $G_{\Q}\cap N^{u}_{\xi}$ (this does not quite equal $\C\cap N^{u}_{\xi}$, but we will show in Proposition \ref{prop: triviality} below that these groups are equivalent for our purposes). $H$ is the group of $\mathbb{Q}$-points of an algebraic group defined over $\mathbb{Q}$ because $G_{\mathbb{Q}}\cap N^{u}_{\xi}$ consists of rational matrices; thus, it admits a Levi decomposition $H = LTU$ over $\mathbb{Q}$, where $L$ is semisimple, $U$ is unipotent, and $T$ is a $\mathbb{Q}$-torus. 

It follows that $\chi(H) = \chi(T)$. Moreover, we have $H_{\Q} = G_{\Q}\cap N^{u}_{\xi}$, and $H_{\Q} = L_{\Q}T_{\Q}U_{\Q}$ because $H$ is defined over $\Q$. 

Write $T = SE$ where the $\mathbb{Q}$-tori $S$ and $E$ are, respectively, $\Q$-split and $\Q$-anisotropic. Then $T_{\Q} = S_{\Q}E_{\Q}$ because $T$ is defined over $\Q$, and therefore $\chi(H_{\Q}) = \chi(T_{\Q}) = \chi(S_{\Q})\chi(E_{\Q})$.

We now make the following key assumption:
\begin{center}
\emph{Suppose that for the character $\chi$ discussed above, we have $\chi(S) = 1$.}
\end{center}
This is the fundamental point that will occupy our discussion below.  If this assumption holds true, we will have by the above $\chi(H_{\Q}) = \chi(T_{\Q})= \chi(S_{\Q})\chi(E_{\Q}) = \chi(E_{\Q})$. Let $K/\Q$ be the smallest number field over which $E$ splits and let $d = [K : \Q]$; observe that $d > 1$ since $E$ is $\Q$-anisotropic. $\chi$ is defined over $K$ (as is any character of $E$), and we have $\chi(E_{\Q})\subset K^{*}$ since $E_{\Q}$ consists of matrices with rational entries. Therefore, letting $N_{K/\Q}$ be the norm map, the composite $\psi = N_{K/\Q}\circ\chi : E_{\Q}\to\Q^{*}$ is well-defined. 

Observe that $E_{\Q}$ has no nontrivial $\Q$-characters; indeed, it is Zariski dense in $E$, so such a character could be extended to a $\Q$-character on all of $E$, contradicting its $\Q$-anisotropy. But $\psi$ is a $\Q$-character of $E_{\Q}$ (indeed, the $\Q$-characters are precisely those characters invariant under the action
\[
\psi^{\sigma}(g) = \sigma(\psi(\sigma^{-1}g))
\]
of $\text{Gal}(K/\Q)$), and therefore must be trivial.

Suppose that $\chi(h)\in\Q$ for some $h\in E_{\Q}$. Then $N_{K/\Q}(\chi(h)) = \chi(h)^{d}$. But $N_{K/\Q}(\chi(h)) = \psi(h) = 1$ by the above, so $\chi(h)^{d} = 1$. Since $d > 1$ and $\chi$ only takes positive real values, it follows that $\chi(h) = 1$. This shows 
\[
 \chi(E_{\Q})\cap\Q^{*} = \{1\}.
\]
Since $\chi(H_{\Q}) = \chi(G_{\Q}\cap N^{u}_{\xi}) = \chi(E_{\Q})$, this yields Thm. \ref{thm: character}, so long as our main assumption above is true. Hence, the proof of Theorem \ref{thm: character} reduces to the following.
\begin{thm}
Suppose $S \subset \xi N_{G}(U)\xi^{-1}$ is a $\Q$-split $\Q$-torus, where the $U$-orbit of $x = \Gamma\xi$ is generic with respect to the Haar measure $\nu_{G}$ on $X$. Then $\chi(S) = \{1\}$. 
\label{thm: torus}
\end{thm}
The proof of this may be divided into two cases.

\subsubsection{$\Gamma$ is an arithmetic lattice in $G$ and $\Gamma\backslash G$ is compact} 
After replacing $G$ by the group $G'$ arising in the definition of arithmetic lattices, we may assume that $\Gamma$ is commensurable with $G_{\Z}$ and the quotient $G_{\Z}\backslash G_{\R}$ is compact. Hence, the theorem of Borel and Harish-Chandra (\cite{BHC} Thm. 11.6) shows that $G$ is a $\Q$-anisotropic group, meaning it contains no nontrivial $\Q$-split torus. The same is therefore also true of any subgroup of $G$, and Theorem \ref{thm: torus} follows immediately in this case.

\subsubsection{$\Gamma$ is an arithmetic lattice in $G$ and $\Gamma\backslash G$ is not compact}
In this case, we may assume by Proposition \ref{prop: commensurable} that $\Gamma$ is commensurable with $G_{\Z}$. This allows us to restrict our attention to rational matrices by the following.
\begin{prop}
 Let $S$ be a subgroup of $\C$ and $\chi: S\to\mathbb{R}^{*}_{>0}$ a character. If $\chi$ is trivial on $S\cap G_{\Q}$, then $\chi$ is trivial on $S$. 
 \label{prop: triviality}
\end{prop}
\begin{proof}
 Since $\Gamma$ is commensurable with $G_{\Z}$, Prop. 4.6 on pg. 206 of \cite{PR} yields $\C = \pi^{-1}((G/N)_{\Q})$, where $N$ is the largest invariant $\Q$-subgroup of compact type. Since $G$ is assumed to have noncompact type, $N$ is the finite subgroup generated by the centers of the $\Q$-simple factors of $G$. Since $N$ is finite, we see that if $g\in S\subset\C$, then $g^{d}\in G_{\Q}$ for some $d\in\mathbb{N}_{>0}$. But then $\chi(g^{d}) = \chi(g)^{d} = 1$ since $\chi$ is assumed to be trivial on $S\cap G_{\Q}$; as $\chi$ only takes positive real values, it follows that $\chi(g) = 1$.     
\end{proof}

In order to prove Theorem \ref{thm: torus} in this case (when $\X$ is non-compact), we will follow the ideas we used in the case of $G=SL_{2}(\mathbb{R})$ above, invoking the fundamental relationship between parabolic subgroups of $G$ and the boundary of the symmetric space $Y = G/K$.
\begin{prop}[\cite{GJT} Prop. 3.8]
 A closed subgroup $P\subset G$ is parabolic if and only if there exists a point $z$ on the visual boundary $Y(\infty)$ of $Y$ such that $P = \text{stab}_{G}(z)$. 
\end{prop}

We now introduce several notions from \cite{AM} in order to relate behavior at the boundary of $Y$ to $\Q$-split tori in $G$, since the latter are the relevant object for Theorem \ref{thm: torus}. Let $y_{0}\in Y$ be any point. 
\begin{defi}
 We call $\eta\in Y(\infty)$ a horospherical limit point for $\Gamma$ if every open horoball based at $\eta$ intersects the orbit $\Gamma\cdot y_{0}$. 
\end{defi}
This definition is independent of the choice of $y_{0}$. Let's illustrate this notion for the case $G = SL_{2}(\R)$ and $\Gamma\subset G$ being any nonuniform arithmetic lattice. We claim that the horospherical limit points for $\Gamma$ acting on $G/K = \mathbb{H}$ are precisely the points $\mathbb{P}^{1}(\R\backslash\Q)$. Indeed, suppose that $\eta\in\mathbb{P}^{1}(\R)$ is \emph{not} a horospherical limit point for $\Gamma$. We may identify $\eta$ with the point at infinity by some element of $G$, under which the horoballs based at $\eta$ become simply the vertical regions $\{z = x + iy\in \mathbb{H} : y > t\}$ for $t\in\R_{>0}$. Thus, the assumption that $\eta$ is not a horospherical limit point for $\Gamma$ implies that for any basepoint $y_{0}\in \mathbb{H}$, there is some $T\in\R_{>0}$ such that the orbit $\Gamma y_{0}$ is contained in the region $\{x+iy: 0 < y \leq T\}$. 

Observe that $\eta$ belongs to $\mathbb{P}^{1}(\Q)$ if and only if its stabilizer $P = \stab_{G}(\eta)\subset G$ contains a $\Q$-split $\Q$-torus. Suppose to the contrary that $P$ contains a $\Q$-anisotropic $\Q$-torus $A$. Then $\eta$ is on the boundary of $Ay_{0}$ for some $y_{0}\in \mathbb{H}$. By the above, the orbit $(A\cap\Gamma)y_{0}$ is bounded below a fixed vertical height in $\mathbb{H}$. However, since $A$ is $\Q$-anisotropic, its group of $\Z$-points $A_{\Z} = A\cap SL_{2}(\Z)$ is a cocompact lattice in $A$, and therefore so too is $A\cap\Gamma$ since $\Gamma$ is commensurable with $SL_{2}(\Z)$. It follows that the orbit $Ay_{0}$ is also vertically bounded in $\mathbb{H}$; but this contradicts the hypothesis that $\eta$ lies on the boundary of $Ay_{0}$. Therefore, $P$ contains a $\Q$-split torus, so $\eta\in \mathbb{P}^{1}(\Q)$.

For the converse, suppose $\eta\in\mathbb{P}^{1}(\Q)$. Then its stabilizer $P$ contains a $\Q$-split torus $S$, and we have the Iwasawa decomposition $G = NSK$ where $N$ is the unipotent radical of $P$. If $\eta$ is a horospherical limit point for $\Gamma$, then it also is one for $SL_{2}(\Z)$, so we may find a sequence of elements $\gamma_{i}\in SL_{2}(\Z)$ such that, writing $\gamma_{i} = n_{i}s_{i}k_{i}$ under the Iwasawa decomposition, we have $\alpha(s_{i})\to\infty$, where $\alpha$ is the root of $S$ in $G$ obtained from the root of the usual $\Q$-split torus defined by
\[
 \alpha\begin{pmatrix} t & 0 \\ 0 & t^{-1} \end{pmatrix} = t^{2}.
\]
Since the stabilizer $P$ of $\eta$ is defined over $\Q$, the $\Z$ points $N_{\Z}$ of its unipotent radical $N$ form a cocompact subgroup of $N$, so there is a compact subset $C\subset N$ such that we may write any element of $N$ as $n = n'c$ where $n'\in N_{\Z}$ and $c\in C$. Therefore, for the sequence $\gamma_{i}$ above we can write 
\[
 \gamma_{i} = n_{i}'c_{i}s_{i}k_{i},
\]
and since $n_{i}\in N_{\Z}\subset SL_{2}(\Z)$, we may simply relabel as needed to see that there is a sequence $\gamma_{i}\in SL_{2}(\Z)$ such that $\gamma_{i} = c_{i}s_{i}k_{i}$, where $\alpha(s_{i})\to\infty$ and $c_{i}$ belongs to the fixed compact set $C\subset N$. Now, since $\alpha(s_{i}) \to\infty$, we have $s_{i}^{-1}ns_{i}\to e$ as $i\to\infty$ for any $n\in N$,  where $e$ is the identity in $G$. Hence, since $C\subset N$ is compact, a diagonal argument applied to the sequence $s_{i}^{-1}c_{i}s_{i}$ shows that, by passing to a subsequence if necessary, we have $s_{i}^{-1}c_{i}s_{i}\to e$ as $i\to\infty$. Therefore, if $i$ is sufficiently large,
\[
 s_{i}^{-1}\gamma_{i} = s_{i}^{-1}c_{i}s_{i}k_{i} \approx k_{i}.
\]
Since $k_{i}\in K$ and $K$ is compact, we may therefore assume by passing to another subsequence that the sequence $s_{i}^{-1}\gamma_{i}$ converges to some $g\in G$. Since $S$ is $\Q$-split, there is some $h\in GL_{2}(\Q)$ such that we may write 
\[
 s_{i} = h\begin{pmatrix} t_{i} & 0 \\ 0 & t_{i}^{-1} \end{pmatrix}h^{-1}
\]
where $t_{i}\to\infty$ as $i\to\infty$. Write 
\[
 h^{-1}g = \begin{pmatrix} a & b \\ c & d\end{pmatrix}.
\]
Then, since $s_{i}^{-1}\gamma_{i}$ converges to $g$, we have for all sufficiently large $i$ 
\[
 \gamma_{i}\approx s_{i}g = h\begin{pmatrix} t_{i} & 0 \\ 0 & t_{i}^{-1} \end{pmatrix}\begin{pmatrix} a & b \\ c & d\end{pmatrix} = h\begin{pmatrix} t_{i}a & t_{i}b \\ t_{i}^{-1}c & t_{i}^{-1}d \end{pmatrix}.
\]
It follows that for all sufficiently large $i$, every entry of the rightmost matrix must be close to the corresponding entry of $h^{-1}\gamma_{i}$. Since $h^{-1}\in GL_{2}(\Q)$, there is an integer $M$ such that $Mh^{-1}\in GL_{2}(\Z)$; then the bottom row of $Mh^{-1}\gamma_{i}$ is close to $(Mt_{i}^{-1}c, Mt_{i}^{-1}d)$ for all large $i$. But $c, d$ and $M$ are fixed, whereas $t_{i}\to\infty$; this implies that the bottom row of the integer matrix $Mh^{-1}\gamma_{i}$ consists only of zeros when $i$ is large enough, a contradiction. 

The ideas we've illustrated in the case of $SL_{2}(\Z)\subset SL_{2}(\R)$ have been generalized by Avramidi and Witte-Morris to yield the following theorem.  For any $\R$-split torus $S\subset G$, there exists a point $y\in Y$ such that $Sy$ is a flat in $Y$ (i.e. a complete, totally geodesic flat submanifold). We say that this flat is $\Q$-split if $S$ is defined over $\Q$ and $\Q$-split. 
\begin{thm}[\cite{AM} Thm. 1.3]
 A point $\eta\in Y(\infty)$ is a horospherical limit point for $\Gamma$ if and only if $\eta$ does not lie on the boundary of any $\Q$-split flat.
 \label{thm: AM}
\end{thm}
%

The normalizer $N_{G}(U)$ is not in general a parabolic subgbroup of $G$, but we do have by a construction of Borel and Tits that there exists a parabolic subgroup $P$ of $G$ such that $N_{G}(U)\subset P$ \cite{BT}. Let $P_{\xi} = \xi P \xi^{-1}$, which is again a parabolic subgroup of $G$, and as before let $N^{u}_{\xi} = \xi N_{G}(U)\xi^{-1}$, so that $N^{u}_{\xi} \subset P_{\xi}$. The following claim is the key to proving Theorem \ref{thm: torus}.
\begin{prop}
 Suppose $\xi\in G$ has the property that the $U-$orbit of the point $\Gamma\xi\in X = \X$ is generic with respect to the Haar measure on $G$. Let 
 \begin{itemize}
 \item $P_{\xi}$ be the parabolic subgroup of $G$ defined above, \\
 \item $A\subset P_{\xi}$ be a maximal $\R$-split torus, \\
 \item $y\in Y$ be a point such that $Ay$ is a flat in $Y$, \\
 \item $\{a^{t}\}$ be a nontrivial one-parameter subgroup of $A$ such that $\{a^{t}\}\subset N^{u}_{\xi}$ and $\{a^{t}\}\not\subset \xi C_{G}(U)\xi^{-1}$, \\
 \item $\eta\in Y(\infty)$ be the endpoint of the geodesic ray $\{a^{t}y\}_{t=0}^{\infty}$, and \\
 \item $\eta'\in Y(\infty)$ be the endpoint of the geodesic ray $\{a^{-t}y\}_{t=0}^{\infty}$. 
 \end{itemize}
 Then at least one of $\eta$ or $\eta'$ is a horospherical limit point for $\Gamma$.
\label{prop: generic}
 \end{prop}
\begin{proof}
 We must show that either any open horoball in $Y$ based at $\eta$ intersects the orbit $\Gamma y$, or that this holds for any open horoball based at $\eta'$. Since $\{a^{t}\}\subset N^{u}_{\xi}$, we have for our character $\chi$ discussed above 
\[
 a^{t}u_{\xi}a^{-t} = u_{\xi}^{\chi(a)^{t}}.
\]
Observe that since $\{a^{t}\}\not\subset \xi C_{G}(U)\xi^{-1} = C_{G}(\xi U\xi^{-1})$, we have either $\chi(a) < 1$ or $\chi(a) > 1$. We may assume that $\chi(a) < 1$, since otherwise this is true with $a^{-1}$ in place of $a$, and the conclusion of the proposition is invariant with respect to this inversion. Therefore
\begin{equation}
 a^{t}u_{\xi}a^{-t} = u_{\xi}^{\chi(a)^{t}}\to e \text{ as }t\to\infty.
\label{eq:goestoe}
\end{equation}
Let $A^{+}$ denote a Weyl chamber of $A$ containing the ray $\{a^{t}\}_{t=0}^{\infty}$, and let
\[
 N = \{u\in G: c^{k}uc^{-k}\to e\text{ as }k\to\infty\text{ for all }c\text{ in the interior of }A^{+}\}.
\]
$N$ is a maximal unipotent subgroup of $G$, and we claim that $\xi U \xi^{-1}\subset N$. Let $\mathfrak{g}$ be the Lie algebra of $G$, and write $u_{\xi} = \exp\nu$ where $\nu\in\mathfrak{g}$ is nilpotent. Let $\mathfrak{a}\subset\mathfrak{g}$ be the toral subalgebra of $\mathfrak{g}$ corresponding to $A$, and let $\Phi$ be a system of roots of $\mathfrak{g}$ with respect to $\mathfrak{a}$ which defines the Weyl chamber $A^{+}\subset G$ (meaning that $A^{+}$ corresponds to a Weyl chamber of $\mathfrak{a}$ for the root system $\Phi$). Then we have the decomposition
\[
\mathfrak{g}=\mathfrak{g}_{0}\oplus\bigoplus_{\alpha\in\Phi}\mathfrak{g}_{\alpha},
\] 
where $\mathfrak{g}_{\alpha}$ are the root spaces
\[
\mathfrak{g}_{\alpha} = \{X\in\mathfrak{g} : [A,X] = \alpha(A)X \text{ for all } A \in \mathfrak{a}\}
\]
and
\[
\mathfrak{g}_{0} = \{X\in\mathfrak{g} : [A,X] = 0 \text{ for all } A \in \mathfrak{a}\}.
\]
Note that $\mathfrak{a}$ need not be a Cartan subalgebra of $\mathfrak{g}$ as it may not be maximal, but that won't be of concern to us. Using this decomposition, we write
\[
\nu = X_{0} + \sum_{\alpha\in\Phi}X_{\alpha},
\]
where $X_{0}\in\mathfrak{g}_{0},X_{\alpha}\in\mathfrak{g}_{\alpha}$. We have
\begin{align*}
 a^{t}u_{\xi}a^{-t} & = a^{t}\exp(\nu)a^{-t} \\
 & = \exp(a^{t}\nu a^{-t}),
\end{align*}
and we see that
\begin{align*}
a^{t}\nu a^{-t} & = a^{t}X_{0}a^{-t} + \sum_{\alpha\in\Phi}a^{t}X_{\alpha}a^{-t} \\
& = \textrm{Ad}(a^{t})X_{0} + \sum_{\alpha\in\Phi}\textrm{Ad}(a^{t})X_{\alpha}.
\end{align*}
Write $a^{t} = \exp(tA)$ with $A\in\mathfrak{a}^{+}$, the Weyl chamber of $\mathfrak{a}$ corresponding to the Weyl chamber $A^{+}\ni a^{t}$. Since $\textrm{Ad}(\exp Y) = \exp(\textrm{ad}_{Y})$, we see that 
\begin{align*}
\textrm{Ad}(a^{t})X_{0} & = \exp(\textrm{ad}_{tA})X_{0} \\
& = X_{0}
\end{align*}
(since $\textrm{ad}_{tA}X_{0} = 0$) and, for every $\alpha\in\Phi$,
\begin{align*}
\textrm{Ad}(a^{t})X_{\alpha} & = \exp(\textrm{ad}_{tA})X_{\alpha} \\
& = e^{t\alpha(A)}X_{\alpha},
\end{align*}
since $\textrm{ad}_{tA}X_{\alpha} = t\alpha(A)X_{\alpha}$. From (\ref{eq:goestoe}), we have $a^{t}u_{\xi}a^{-t}\to e$, and therefore 
\[
a^{t}\nu a^{-t} = X_{0} + \sum_{\alpha\in\Phi}e^{t\alpha(A)}X_{\alpha}\to 0 \text{ as }t\to\infty.
\]
It follows that $X_{0} = 0$ and $\alpha(A) < 0$ whenever $X_{\alpha}\neq 0$. Since $A$ is an element of the Weyl chamber $\mathfrak{a}^{+}$, it follows that
\[
\left\{\alpha\in\Phi : X_{\alpha}\neq 0\right\}\subset \left\{\alpha\in\Phi : \alpha(C) < 0 \text{ for all } C\in\mathfrak{a}^{+}\right\}.
\]
Thus, if $c	 = \exp C \in G$ is any element of the interior of $A^{+}$, so that $C\in\mathfrak{a}^{+}$, we have
\begin{align*}
c^{k}u_{\xi}c^{-k} & = \exp\sum_{\alpha\in\Phi}\textrm{Ad}(c^{k})X_{\alpha}\\
& =  \exp\sum_{\alpha\in\Phi}e^{k\alpha(C)}X_{\alpha} \to \exp 0 = e \text{ as } k\to\infty.
\end{align*}
It follows that $\xi U\xi^{-1}\subset N$, as claimed. 

Now, let $A_{\perp}$ denote the complement in $A$ of $\{a^{t}\}$ with respect to the Killing form. Then for any $s\in\R_{\geq 0}$, the set
 \[
  \mathcal{O}_{s} = \bigcup_{t > s}NA_{\perp}a^{t}y \subset Y
 \]
is an open horoball based at $\eta$, and any such horoball may be written in this form (\cite{AM} Lemma 2.5).

For any $s\in\R$, the set $V_{s} = \cup_{t > s}N_{\xi}A_{\perp}a^{t}\tilde{y}K\subset G$ is open in $G$, where $\tilde{y}\in G$ is chosen so that its coset in $G/K$ equals $y$. Now, since the $U$ orbit of $\Gamma\xi$ is dense in $\X$, the set $\Gamma\xi U$ is dense in $G$. Therefore, the set $U\xi^{-1}\Gamma$ is also dense in $G$, and so too is $\xi U \xi^{-1}\Gamma\tilde{y}$. Hence, the latter set must intersect the open set $V_{s}$, so there exists some $t > s$ such that
\[
 NA_{\perp}a^{t}\tilde{y}K\cap\xi U \xi^{-1}\Gamma\tilde{y}\neq\emptyset. 
\]
But by the preceding discussion, we have $\xi U \xi^{-1}\subset N$; thus the above implies
\[
  NA_{\perp}a^{t}\tilde{y}K\cap\Gamma\tilde{y}\neq\emptyset,
\]
and since $t > s$, this shows that $\mathcal{O}_{s}\cap\Gamma y \neq\emptyset$, proving the claim.

\end{proof}

We are now in a position to prove Theorem \ref{thm: torus}.

\begin{proof}[Proof of Theorem \ref{thm: torus}]
Suppose $S\subset \xi N_{G}(U)\xi^{-1}$ is a nontrivial $\Q$-split $\Q$-torus, but that $\chi(S)\neq \{1\}$. Then the definition of $\chi$ implies there is a one-parameter subgroup $\{a^{t}\}\subset S$ (so that $\{a^{t}\}\subset N^{u}_{\xi}$) such that $\{a^{t}\}\not\subset \xi C_{G}(U)\xi^{-1}$. Moreover, since $S\subset \xi N_{G}(U)\xi^{-1}\subset P_{\xi}$ and $S$ is a $\mathbb{Q}$-split and thus $\mathbb{R}$-split torus in $G$, and since $P_{\xi}$ is parabolic, there is a maximal $\mathbb{R}$-split torus $A\subset P_{\xi}$ of $G$ such that $S\subset A$, and in particular $\{a^{t}\}\subset A$. Choose $y\in Y$ such that $Ay$ is a flat in $Y$. Let $\eta\in Y(\infty)$ be the endpoint of the geodesic ray $\{a^{t}y\}_{t=0}^{\infty}$, and $\eta'$ the endpoint of $\{a^{-t}y\}_{t=0}^{\infty}$. Then by Proposition \ref{prop: generic}, either $\eta$ or $\eta'$, both of which are on the boundary of the $\mathbb{Q}$-split flat $Sy$, is a horospherical limit point for $\Gamma$. But this contradicts Theorem \ref{thm: AM}.    
\end{proof}

This proves Theorem \ref{thm: character} when $\Gamma$ is irreducible, but we wish to consider all lattices in the statement of Theorem \ref{thm: disjoint}. Thus, suppose that $\Gamma$ is any lattice in $G$. The space $\Gamma\backslash G$ is finitely covered by a homogeneous space of a semisimple group which is totally noncompact and has trivial center; since it clearly suffices to prove Theorem \ref{thm: disjoint} for a cover of $\Gamma\backslash G$, we may assume that $G$ itself has trivial center. Then there is a direct product decomposition $G = G_{1}\times G_{2}\times\cdots\times G_{r}$ such that each $G_{i}$ is semisimple and $\Gamma$ is commensurable to $\Gamma_{1}\times\cdots\times \Gamma_{r}$, where $\Gamma_{i} = \Gamma\cap G_{i}$ is irreducible in $G_{i}$ for each $i$ (\cite{WM2} Proposition 4.28). 

Let $\Gamma' = \Gamma\cap\Gamma_{1}\times\cdots\times \Gamma_{r}$, so that $X' = \Gamma'\backslash G$ is a finite cover of both $X = \Gamma\backslash G$ and of $\Gamma_{1}\backslash G_{1}\times\cdots\times \Gamma_{r}\backslash G_{r}$. Let $\{x_{1}',\dots,x_{m}'\}\subset X'$ be the fiber of $X'\to X$ over $x$. Since the $u$-orbit of $x$ is generic for the $G$-invariant probability measure on $X$, there must exist some $1\leq j\leq m$ such that the $u$ orbit of $x_{j}'$ is generic for the $G$-invariant probability measure on $X'$. The $u$-orbit of the image of this point in $\Gamma_{1}\backslash G_{1}\times\cdots\times \Gamma_{r}\backslash G_{r}$ is therefore generic for $\nu_{G_{1}}\times\cdots\times\nu_{G_{r}}$, using the obvious notation. This yields points $x_{i}\in\Gamma_{i}\backslash G_{i}$ for $i=1,\dots,r$, such that the $u_{i}$-orbit of $x_{i}$ is generic for $\nu_{G_{i}}$, where $u_{i}$ is the projection of $u$ to $G_{i}$.

Since the $u_{i}$-orbit of $x_{i}$ is generic for $\nu_{G_{i}}$, and since $\Gamma_{i}$ is an irreducible lattice in $G_{i}$, we find from Theorem \ref{thm: character} that the set $P_{i}$ of pairs of distinct primes $p\neq q$ for which there exists a nontrivial joining of $(X_{i},u_{i}^{p},\nu_{G_{i}})$ with $(X_{i},u_{i}^{q},\nu_{G_{i}})$, supported on the $(u_{i}^{p},u_{i}^{q})$-orbit of $(x_{i},x_{i})$, is finite for $i=1,\dots,r$.

Suppose that $\nu\neq \nu_{G}\times \nu_{G}$, where $\nu$ is the joining appearing in (\ref{eq: semisimple correlations}). Then there is a nontrivial joining $\nu'$ of $(X',u^{p},\nu_{G})$ with $(X',u^{q},\nu_{G})$ that projects to $\nu$, and whose image on $\Gamma_{1}\backslash G_{1}\times\cdots\times \Gamma_{r}\backslash G_{r}$ is nontrivial and supported on the $u$-orbit of $((x_{1},\dots,x_{r}),(x_{1},\dots,x_{r}))$. There must then exist some $1\leq i\leq r$ such that the projection of this joining to $X_{i}\times X_{i}$ is nontrivial. Hence, the set $P$ of pairs of distinct primes $p\neq q$ for which there exists a nontrivial joining of $(X,u^{p},\nu_{G})$ with $(X,u^{q},\nu_{G})$, supported on the $(u^{p},u^{q})$-orbit of $(x,x)$, is contained in $P_{1}\cup\cdots \cup P_{r}$, and thus must be finite by the above. This proves Theorem \ref{thm: character} for all lattices $\Gamma$, which we now show implies Theorem \ref{thm: disjoint}.

\begin{proof}[Proof of Theorem \ref{thm: disjoint}]
Let $c = \int_{\Gamma\backslash G}f\,\textrm{d}\nu_{G}$. Since
\[
\sum_{n\leq N}\mu(n)f(xu^{n}) = \sum_{n\leq N}\mu(n)(f(xu^{n}) - c) + c\sum_{n\leq N}\mu(n) = \sum_{n\leq N}\mu(n)(f(xu^{n}) - c) + o(N)
\]
by the prime number theorem, we may assume without loss of generality that $c = 0$. By (\ref{eq: semisimple correlations}) and Theorem \ref{thm: character}, we then have for all but finitely many $p\neq q$
\[
\lim_{N\to\infty}\frac{1}{N}\sum_{n\leq N}f(xu^{pn})\overline{f(xu^{qn})} = \left(\int_{\X}f\,\textrm{d}\nu_{G}\right)^{2} = 0,
\]
and Theorem \ref{thm: disjoint} now follows from Theorem \ref{thm: criterion}. 
\end{proof}

\section{The general case}\label{sec: levi}
We now consider all possibilities for the group $H$ appearing in Theorem \ref{thm: equidistribution}, rather than only semisimple ones. Indeed, letting $x = \Gamma\xi$ with $\xi\in G$, this theorem yields a closed subgroup $C\subset G\times G$ such that $\{(u^{pt},u^{qt})\}_{t\in\mathbb{R}}\subset C, \, \Gamma_{C}^{\xi} := \widetilde{\xi}^{-1}(\Gamma\times\Gamma)\widetilde{\xi}\cap C$ is a lattice in $C$ (where $\widetilde{\xi} = (\xi,\xi)\in G\times G$), and 
\begin{equation}
 \lim_{N\to\infty}\frac{1}{N}\sum_{n=0}^{N-1}f(xu^{pn})f(xu^{qn}) = \int_{X\times X}f\otimes f\,\mathrm{d}\nu_{C}
\label{eq: nongeneric correlations}
\end{equation}
for any bounded continuous function $f$ on $X$, where $\nu_{C}$ is a $C$-invariant probability measure supported on $(x,x)C$. At the same time, there exists a closed subgroup $H\subset G$ such that $\{u^{t}\}_{t\in\mathbb{R}}\subset H, \Gamma_{H}^{\xi} := \xi^{-1}\Gamma\xi\cap H$ is a lattice in $H$ and
\[
 \lim_{N\to\infty}\frac{1}{N}\sum_{n=0}^{N-1}f(xu^{n}) = \int_{X}f\,\mathrm{d}\nu_{H}
\]
for any bounded continuous function $f$ on $X$, where $\nu_{H}$ is an $H$-invariant probability measure supported on $xH$. It follows that $\nu_{C}$ is a joining of the systems $(X,u^{p},\nu_{H})$ and $(X,u^{q},\nu_{H})$. Observe that each of these systems is conjugate (in the measure-theoretic sense) to the corresponding system on the smaller homogeneous space $\Gamma_{H}^{\xi}\backslash H$; this is provided by the map 
\begin{gather*}
\Gamma\backslash G \to \Gamma_{H}^{\xi}\backslash H \\
\Gamma \xi h \mapsto \Gamma_{H}^{\xi}h,
\end{gather*}
which is well-defined and defined $\nu_{H}$-almost everywhere, since this measure is supported on $xH$. The image of the measure $\nu_{C}$, which lives on the product $X\times X$, under the product of this conjugacy with itself, is then a joining of the systems $(\Gamma_{H}^{\xi}\backslash H, u^{p}, \nu_{H})$ and $(\Gamma_{H}^{\xi}\backslash H, u^{q}, \nu_{H})$. We will refer to this joining as $\lambda$, and denote $X_{H} := \Gamma_{H}^{\xi}\backslash H$. We will also abuse notation and write $\Gamma_{H}^{\xi} = \Gamma$ from now on. By the second part of Theorem \ref{thm: equidistribution} , we may assume that the systems $(X_{H} , u, \nu_{H})$ and $(X_{H} \times X_{H}, u^{p}\times u^{q},\lambda)$ are ergodic.

Now, since the identity component $H^{0}$ has finite index in $H$, it does no harm to assume that $H$ is connected. Furthermore, on passing to the universal cover, we may assume that $H$ is simply connected as well. This is justified by the obvious fact that the statement of Theorem \ref{thm: main} holds for $\Gamma\backslash H$ as soon as it holds for any topological cover thereof. 
%

We now consider the real Levi decomposition of the connected, simply connected real Lie group $H$ \cite{PR}: we may write $H$ as a semi-direct product $H = L\rtimes R$, where $L$ is a maximal semisimple subgroup of $H$, and $R$ is its solvable radical (the maximal connected normal solvable subgroup of $H$). $R$ is the kernel of the projection $p : H\to L$, and both $L$ and $R$ are connected and simply connected. We wish to relate the unipotent flow on $\Gamma\backslash H$ to one on a homogeneous space of $L$, to which end we require the following fact.

\begin{prop}[\cite{Sta2},\cite{Wi}]
Suppose that $H$ is a connected Lie group with Levi decomposition $H = L\rtimes R$, and $\Gamma\subset H$ is a lattice such that $\Gamma\backslash H$ supports an ergodic unipotent translation. Then $L$ is totally non-compact (i.e. it has no nontrivial compact factors). 
\end{prop}

It follows from this and Wang's theorem (\cite{Sta} Theorem E.11) that the projection of $\Gamma$ into $L$ under $p$ is a lattice $\Gamma_{L}$ in $L$, and $R\cap \Gamma$ is a lattice $\Gamma_{R}$ in $R$. Let $X_{L} = \Gamma_{L}\backslash L$, and let $\nu$ be the image of $\lambda$ under the product projection $X_{H} \times X_{H} \to X_{L}\times X_{L}$.

 $\nu$ is an ergodic joining of the systems $(X_{L}, u_{L}^{p},\nu_{L})$ and $(X_{L}, u_{L}^{q},\nu_{L})$, where $u_{L} = p(u) \in L$ and $\nu_{L}$ is the $L$-invariant Borel probability measure on $X_{L}$. Moreover, $\nu$ is supported on the closure of the $(u_{L}^{p},u^{q}_{L})$-orbit of a point $(x_{L},x_{L})$, where the $u_{L}$ orbit of $x_{L}$ in $X_{L}$ is uniformly distributed with respect to $\nu_{L}$ (indeed, the topological density of the $u$-orbit of $x$ in $X_{H}$ implies the same of its image under the continuous surjection $X_{H}\to X_{L}$). 
 
 It now follows from our analysis in the previous section (specifically, Theorem \ref{thm: character}) that either $\nu = \nu_{L}\times\nu_{L}$, or the pair $(p,q)$ belongs to a finite set determined by the point $x$. The latter case has no effect on the disjointness criterion (Theorem \ref{thm: criterion}), so we may assume in what follows that
 \begin{equation}
 \nu = \nu_{L}\times\nu_{L}.
 \label{eq: trivial}
 \end{equation}
%
%

Thus, we have an ergodic joining $\lambda$ of $(X_{H}, u^{p}, \nu_{H})$ and $(X_{H}, u^{q}, \nu_{H})$ whose projection to $X_{L}\times  X_{L}$ is the product measure $\nu_{L}\times \nu_{L}$. In order to lift this information to the level of the group $H$, where the Levi decomposition $H = L\rtimes R$ may be directly used, we need the full generality of Ratner's joinings theorem. 

To state this, we require some terminology. Suppose that $H$ a connected Lie group, $\Gamma$ a lattice in $H$, $u_{1}, u_{2}\in H$ unipotent elements, and $\nu$ the $H$-invariant Borel probability measure on $X = \Gamma\backslash H$. Let $\lambda$ be an ergodic algebraic joining of $(X, u_{1}, \nu)$ with $(X, u_{2},\nu)$ and $\Lambda = \stab_{H\times H}(\lambda)$, so that there exists some $x(\lambda)\in X\times X$ such that $\lambda(x(\lambda)\Lambda) = 1$ (by Ratner's measure classification theorem - \cite{R1} Theorem 1). Then the subgroups $\Lambda_{1},\Lambda_{2}\subset H$ defined by 
\begin{equation}
\Lambda_{1} = \{h\in H: (h,e)\in \Lambda\}, \hspace{0.3in} \Lambda_{2} = \{h\in H : (e,h)\in\Lambda\}
\label{eq: subgroups}
\end{equation}
are closed normal subgroups of $H$. For $z\in X$, let
\[
\xi_{\lambda}(z) = \{y \in X : (z,y) \in x(\lambda)\Lambda\}.
\]
This is called the $z$-fiber of $\lambda$.

\begin{thm}[\cite{R1} Theorem 2]
Let all notation and assumptions be as in the above paragraph. Then there is a $c\in H$ and a continuous, surjective homomorphism $\alpha: H \to H/\Lambda_{2}$ with kernel $\Lambda_{1}$ and $\alpha(u_{1}) = u_{2}\Lambda_{2}$, such that
\[
\xi_{\lambda}(\Gamma h) = \left\{\Gamma c\beta_{i}\alpha(h) : i=1,\dots, n\right\}
\]
for all $h\in H$, where the intersection $\Gamma_{0} = \alpha(\Gamma)\cap c\Gamma c^{-1}\Lambda_{2}$ is of finite index in $\alpha(\Gamma)$ and $c\Gamma c^{-1}\Lambda_{2}, n = |\Gamma_{0}\backslash \alpha(\Gamma)|$, and $\alpha(\Gamma) = \{\Gamma_{0}\beta_{i} : i=1,\dots, n\}$. 
\label{thm: joinings}
\end{thm}
Recall that, in our setting, $\lambda$ is an ergodic joining of $(X_{H}, u^{p},\nu_{H})$, with $(X_{H}, u^{q},\nu_{H})$, and its pushforward to $X_{L}\times X_{L}$ under the map induced by the projection $p\times p: H\times H\to L\times L$ is the product measure $\nu_{L}\times \nu_{L}$. We now process Ratner's description of joinings in the following way. 

\begin{prop}
Let $\Lambda_{1},\Lambda_{2}\subseteq H$ be the normal subgroups defined in eq. (\ref{eq: subgroups}). Then we have $p(\Lambda_{1}) = p(\Lambda_{2}) = L$.
\label{prop: projection}
\end{prop}
\begin{proof}
By Ratner's measure classification theorem, we have $\lambda(x_{0}\Lambda) = 1$ for some $x_{0}\in X_{H}\times X_{H}$, and it follows that $(\pi\times\pi)_{*}\lambda((\pi\times\pi)(x_{0}\Lambda)) = 1$. Clearly, $(\pi\times\pi)(x_{0}\Lambda) = ((\pi\times\pi)x_{0})(p\times p)\Lambda$. Since $(\pi\times\pi)_{*}\lambda = \nu_{L}\times \nu_{L}$, it follows that we must have $((\pi\times\pi)x_{0})(p\times p)\Lambda = X_{L}\times X_{L}$, which yields $(p\times p)\Lambda = L\times L$.

Given $h\in H$, we write $\alpha(h) = \widetilde{h}\Lambda_{2}$, where $\alpha: H\to H/\Lambda_{2}$ is the surjective homomorphism of Ratner's joinings theorem above. Since $\Lambda_{2}$ is a closed, normal subgroup of $H$, $H/\Lambda_{2} := H_{2}$ is a connected Lie group in its own right, and we let $p_{2}: H_{2}\to L/p(\Lambda_{2})$ be the projection induced by $p: H\to L$. Observe that we may identify $L/p(\Lambda_{2}) := L_{2}$ as a Levi subgroup of $H_{2}$.

Now, let $h\in H$, and let $x\in\xi_{\lambda}(\Gamma h)$ be a point in the $\Gamma h$-fiber of $\lambda$. Then by Ratner's joinings theorem, there exists an index $i\in \{1,\dots,n\}$ such that $x \in \Gamma c\beta_{i}\alpha(h) = \Gamma c\beta_{i}\widetilde{h}\Lambda_{2}$. It follows that
\[
\pi_{L}(\xi_{\lambda}(\Gamma h)) = \{\pi_{L}(\Gamma c\beta_{i}\widetilde{h}\Lambda_{2}) : i=1,\dots,n\}= \{\Gamma_{L}c_{L}\delta_{i}p_{2}(\widetilde{h}\Lambda_{2}) : i=1,\dots,n\}.
\]
where $c_{L} = \pi_{L}(c)$ and $\delta_{i} = p(\beta_{i})$. Now, observe that $\alpha(R)$ is a connected, normal, solvable subgroup of $H_{2}$, and therefore that $p_{2}$ is trivial on $\alpha(R)$ since $L_{2}$ is semisimple, being a quotient of the semisimple group $L$. It follows that if we let $h = \ell r$ with $\ell\in L$ and $r\in R$, then
\begin{align*}
p_{2}(\alpha(h)) & = p_{2}(\alpha(\ell r))\\
& = p_{2}(\alpha(\ell)\alpha(r)) \\
& = p_{2}(\alpha(\ell)) \\
& = p(\widetilde{\ell})p(\Lambda_{2}).
\end{align*} 
Therefore,
\[
\pi_{L}(\xi_{\lambda}(\Gamma h)) =  \{\Gamma_{L}c_{L}\delta_{i}p(\widetilde{\ell})p(\Lambda_{2}): i=1,\dots,n\}.
\]
On the other hand, we have from above that $(p\times p)(\Lambda) = L\times L$, from which it follows that $\pi_{L}(\xi_{\lambda}(\Gamma h)) = X_{L}$ for any $h\in H$. Indeed, if we let $w\in X_{L}$, then there is some $z\in X_{H}$ such that $(\pi\times\pi)(\Gamma h, z) = (y,w)$ and $(\Gamma h,z)\in\supp(\lambda) = x_{0}\Lambda$. It's then clear that $\pi(z) = w$ and $z\in\xi_{\lambda}(\Gamma h)$. We have shown that
\[
\{\Gamma_{L}c_{L}\delta_{i}p(\widetilde{\ell})p(\Lambda_{2}) : i=1,\dots,n\} = X_{L},
\]
and since $p(\Lambda_{2})$ is a closed subgroup of $L$ while $p(\widetilde{\ell})\in L$ and $n\in\mathbb{N}$ are fixed, it follows that $p(\Lambda_{2}) = L$, as desired (otherwise, we would have written $X_{L}$ as a finite union of lower-dimensional submanifolds). We immediately get the same conclusion for $\Lambda_{1}$, since by Ratner's joinings theorem, the groups $H/\Lambda_{1}$ and $H/\Lambda_{2}$ are isomorphic (under $\alpha$). Hence, the triviality of the Levi factor $L/p(\Lambda_{2})$ of $H_{2}$ implies that $L/p(\Lambda_{1})$ must be trivial as well.

\end{proof}

The following easy lemma will allow us to use this proposition to relate the action of $(u^{p}, u^{q})$ to the action of $(u_{L}^{p}, u_{L}^{q})$, where $u_{L} = p(u)\in L$, and to that of $(u_{R}^{p},u_{R}^{q})$, where $u_{R}\in R$ is the unique element such that $u = u_{L}u_{R}$.

\begin{lem}
Let $H$ be a connected Lie group with Levi decomposition $H = L\rtimes R$. If $N\subset H$ is a closed, connected, normal subgroup such that $p(N) = L$, then $L\subset N$.  
\end{lem}
\begin{proof}
Since $N$ is closed in $H$, it is a connected Lie group in its own right, and thus admits a Levi decomposition $N = S\rtimes Q$, where $Q$ is the solvable radical of $N$. Observe that $R\cap N$ is a connected, normal, solvable subgroup of $N$, and therefore $R\cap N\subseteq Q$ since $Q$ is the largest such subgroup. At the same time, since $p|_{N} : N\to L$ is surjective, the image $p(Q)\subset L$ is a normal, connected, solvable subgroup of $L$, and therefore must be trivial since $L$ is semisimple. It follows that $Q\subset R =\ker(p)$, and therefore $Q\subseteq R\cap N$. Since we already had the reverse containment, this shows that $Q = R\cap N$.

It follows that $N = S\rtimes (R\cap N)$, with $S$ a Levi subgroup of $N$. Therefore, we have $p(N) = p(S)$, and since $p(N) = L$ by hypothesis, we see that $\dim S \geq \dim L$. Since $L$ is a maximal semisimple subgroup of $H$, and since any two such subgroups are conjugate in $H$, it follows that $S$ too is a Levi subgroup of $H$. Since $N$ is normal in $H$, it follows that $N$ contains all Levi subgroups of $H$, and in particular $L\subset N$. 
\end{proof}

Combining this lemma with the proposition before it, we find that $L\subset \Lambda_{i}$ for $i=1,2$. In particular, we have $u_{L}^{p}\in \Lambda_{1}$, and $u_{L}^{q}\in \Lambda_{2}$. Let $u_{R}^{(p)} = u_{L}^{-p}u^{p}$ and $u_{R}^{(q)} = u_{L}^{-q}u^{q}$. Then $u_{R}^{(p)},u_{R}^{(q)}\in R$, and both are unipotent elements. We find 
\begin{equation}
(u_{L}^{p},u_{L}^{q}), \, (u_{R}^{(p)}, u_{R}^{(q)})\in \Lambda = \stab_{H\times H}\lambda.
\label{eq: coarse levi}
\end{equation}
The presence of $(u_{R}^{(p)},u_{R}^{(q)})$, rather than a pair of genuine powers of $u_{R}$, is unappealing, and can be easily rectified. Indeed, it is easy to see that
\begin{align*}
u_{R}^{(p)} & = \prod_{k=0}^{p-1}u_{L}^{-k}u_{R}u_{L}^{k} \\ 
& = u_{R}^{p}\cdot\prod_{k=0}^{p-1}u_{R}^{-1}u_{L}^{-k}u_{R}u_{L}^{k}.
\end{align*}
Since $\Lambda_{i}$ is normal in $H$ and $L\subset \Lambda_{i}$, we have $u_{L}^{k}\in\Lambda_{i}$ and $u_{R}^{-1}u_{L}^{-k}u_{R}\in\Lambda_{i}$ for all $k\in\mathbb{Z}, i=1,2$. Therefore
\[
\prod_{k=0}^{p-1}u_{R}^{-1}u_{L}^{-k}u_{R}u_{L}^{k}\in\Lambda_{1}\text{ and } \prod_{k=0}^{q-1}u_{R}^{-1}u_{L}^{-k}u_{R}u_{L}^{k}\in\Lambda_{2},
\]
and so we find
\[
\left(\prod_{k=0}^{p-1}u_{R}^{-1}u_{L}^{-k}u_{R}u_{L}^{k}, \prod_{k=0}^{q-1}u_{R}^{-1}u_{L}^{-k}u_{R}u_{L}^{k}\right)\in\Lambda.
\]
But, on combining this with the above expression for $u_{R}^{(p)}$ and (\ref{eq: coarse levi}), we see that we also have
\[
(u_{R}^{p},u_{R}^{q}) = (u_{R}^{(p)},u_{R}^{(q)})\cdot\left(\prod_{k=0}^{p-1}u_{R}^{-1}u_{L}^{-k}u_{R}u_{L}^{k}, \prod_{k=0}^{q-1}u_{R}^{-1}u_{L}^{-k}u_{R}u_{L}^{k}\right)^{-1}\in\Lambda.
\]

We pause to formally record these points.
\begin{prop}
Let all notation be as above for the connected, simply connected Lie group $H$. Then the following holds for all but finitely many pairs of primes $p\neq q$: if $\lambda$ is an ergodic joining of the systems $(X_{H},u^{p},\nu_{H})$ and $(X_{H},u^{q},\nu_{H})$ which is supported on the closure of the $(u^{p},u^{q})$ orbit of $(\Gamma e,\Gamma e)$, then we have both
\[
(u_{L}^{p},u_{L}^{q})\in\stab_{H\times H}\lambda
\]
and
\[
(u_{R}^{p},u_{R}^{q})\in\stab_{H\times H}\lambda.
\]
\label{prop: levi parts}
\end{prop}

The short exact sequence
\[
1\xrightarrow{} R \xrightarrow{}H\xrightarrow{p} L\xrightarrow{} 1
\]
gives rise to a fibration
\[
1\xrightarrow{} X_{R}\xrightarrow{} X_{H}\xrightarrow{\pi} X_{L}\xrightarrow{}1,
\]
hence also a fibration
\[
1\xrightarrow{} X_{R}\times X_{R}\xrightarrow{} X_{H}\times X_{H}\xrightarrow{\pi\times\pi} X_{L}\times X_{L}\xrightarrow{}1.
\]

For every $y \in X_{L}$, we may identify the fiber $\mathcal{F}_{y} := \pi^{-1}(\{y\})$ with $X_{R} = \Gamma_{R}\backslash R$: fix once and for all a Borel measurable section $\sigma: X_{L}\to L$ of the projection $L\to X_{L}$, the existence of which is guaranteed by the results of \cite{Mac}. This section need not be differentiable or even continuous, and the reader may rightly worry that this will cause problems in our analysis, since Ratner's theorems only apply as stated to continuous functions, and moreover we will be using aspects of harmonic analysis on the nilmanifold $X_{R}$ that only apply to continuous functions. However, the reason we are safe to work with a section that is only measurable is that we have, at this point, already employed the continuity of $f$ and the genericity of the point $x$ for $\nu_{H}$ to obtain the joining $\lambda$. All of our work will now proceed relative to $\lambda$ and the various Haar measures involved, at which level the distinction between continuity and mere measurability is not harmful.  

That said, the continuity of $f$ (and in fact its differentiablility, which we will assume shortly), will be essential to our arguments involving harmonic analysis, since we will be examining the fine properties of $f$ in the $X_{R}$ fibers mentioned above. The point is that it is only the properties of $f$ in $\nu_{L}$-almost all of these fibers that are relevant, once we take averages against the various pertinent measures. We will be clear in what follows at which points we are exploiting the continuity or differentiability of $f$, versus the ability to restrict our attention to properties that only hold almost everywhere with respect to some invariant measure of interest.  

 We will write $\sigma(y) = \widetilde{y}\in L$ from here on. Then the map 
\begin{align*}
\varphi_{y}: X_{R} & \to\mathcal{F}_{y} \\
\Gamma_{R}r & \mapsto \Gamma r \widetilde{y}
\end{align*}
is clearly well-defined and is a homeomorphism, which explicitly identifies the fibers of $\pi_{L}$ with $X_{R}$.

The following useful fact, due to Witte Morris, will allow us to apply an inductive argument to prove M\"{o}bius disjointness, along the lines of \cite{GT} and \cite{Z}.

\begin{prop}[\cite{Wi} Proposition 2.6]
Suppose that $H$ is a connected Lie group and $\Gamma\subset H$ is a lattice, such that $\Gamma\backslash H$ supports an ergodic unipotent translation. Then $rad H$ is nilpotent.
\label{prop: nilpotent}
\end{prop}

Since we are beginning with the assumption that $\Gamma\backslash H$ supports an ergodic unipotent translation, we see that $R$ is nilpotent. Our argument will proceed by induction on the step $d$ of $R$; thus the base case $d=1$ is that in which $R$ is abelian.

The following version of the Mautner phenomenon, due to Ratner, allows us to reduce to the case in which the action of $u_{R}$ on $X_{R}$ is ergodic.
\begin{prop}[\cite{R2} Proposition 1.5]
Let $H$ be a connected, simply connected Lie group whose radical is nilpotent and which has no compact semisimple factors, and let $\Gamma$ be a lattice in $H$. Let $U$ be a connected unipotent subgroup of $H$ and $N$ the smallest closed connected normal subgroup of $H$ such that $U\subset N$ and $N\cap\Gamma$ is a lattice in $N$.Then $\nu_{H}$-almost every ergodic component of the action of $U$ on $(\Gamma\backslash H, \nu_{H})$ is an $N$-orbit.
\label{prop: ergodic components}
\end{prop}

Taking $U = \{\exp(tv_{R}):t\in\mathbb{R}\}$, where $v_{R} = \log{u_{R}}$, the ergodic components of the action of $U$ on $(X_{H},\nu_{H})$ agree with those of the action of $u_{R}$. Since $u_{R}\in R$, $R$ is a closed, connected, normal subgroup of $H$ by construction, and $\Gamma\cap R$ is a lattice in $R$ by the above remarks, we see that $N\subset R$, where $N$ is the subgroup of the proposition. Replacing $H$ with the smaller group $L\rtimes N$, we may assume that $\nu_{H}$-a.e. ergodic component of the action of $u_{R}$ on $(X_{H},\nu_{H})$ is an $R$-orbit. In particular, we may assume that the action of $u_{R}$ on $X_{R}$ is ergodic (for the Haar measure $\nu_{R}$), and that the ergodic decomposition of $\nu_{H}$ for the $u_{R}$ action agrees with the disintegration of $\nu_{H}$ over $\pi: X_{H}\to X_{L}$:
\[
\nu_{H} = \int_{X_{L}}\nu_{y}\,\textrm{d}\nu_{L}(y),
\]
where $\nu_{y}$ is supported on $\mathcal{F}_{y}$ and satisfies $(\varphi_{y}^{-1})_{*}\nu_{y} = \nu_{R}$, for $\nu_{L}$-a.e. $y\in X_{L}$.

This ergodicity now implies the following.
\begin{prop}
For $\nu_{L}$-a.e. $y\in X_{L}$, the action of $\y u_{R}\y^{-1}$ on $X_{R}$ is ergodic relative to $\nu_{R}$, and therefore uniquely and totally ergodic.
\label{prop: haar ergodic} 
\end{prop} 
\begin{proof}
By Proposition \ref{prop: nilpotent}, $R$ is nilpotent. Thus we may apply a theorem of Leon Green (\cite{AGH}), which asserts that a rotation on a nilmanifold is ergodic if and only if its projection to the associated horizontal torus is ergodic. Thus, it suffices to assume that $X_{R}$ is a torus. Let 
\[
E = \left\{\ell\in L :\text{the action of $\ell u_{R}\ell^{-1}$ on $(X_{R},\nu_{R})$ is not ergodic}\right\}.
\]
By the ergodicity criterion for a rotation on a torus, we have
\[
E = \bigcup_{\substack{\chi\in\widehat{X}_{R} \\ \chi\neq \chi_{0}}}E_{\chi}
\]
where
\[
E_{\chi} = \left\{\ell\in L : \chi(\ell u_{R}\ell^{-1}) = 1\right\}
\]
and $\chi_{0}$ is the trivial character. Observe that $E_{\chi}$ is an analytic subvariety of $L$. Moreover, no $E_{\chi}$ contains the identity element of $L$, since $u_{R}$ acts ergodically. Therefore, each $E_{\chi}$ is a \emph{proper} analytic subvariety of $L$, and thus has Haar measure 0 in $L$. Since $\widehat{X_{R}}$ is countable, it follows that the Haar measure of $E$ is 0 as well. But it's then immediate that
\[
\nu_{L}\left(\left\{y\in X_{L} : \text{the action of $\y u_{R}\y^{-1}$ on $(X_{R},\nu_{R})$ is not ergodic}\right\}\right) = 0,
\]
as desired.

Finally, the statements of unique and total ergodicity are true in general for ergodic nilrotations (cf. \cite{Fra} page 10).
\end{proof}

Notice that for any $r,r' \in R$ we have 
\begin{equation}
\varphi_{y}(\Gamma_{R} rr') = \Gamma rr'\widetilde{y} = \Gamma r\y (\y^{-1}r'\y) = \varphi_{y}(\Gamma_{R}r)\y^{-1}r'\y,
\label{eq: equivariance}
\end{equation}
and $\y^{-1}r\y\in R$ since $R$ is normal in $H$. 

Now observe that for any point $y\in X_{L}$ and any $r\in R$, we have
\begin{equation}
\varphi_{y}^{-1}\left(\Gamma r\y u_{R}^{p}\right) = \varphi_{y}^{-1}\left(\Gamma r\left(\y u_{R}^{p} \y^{-1}\right)\y\right) = \Gamma_{R} r \left(\y u_{R}^{p} \y^{-1}\right).
\label{eq: fiber action}
\end{equation}

Let
\begin{equation}
\lambda = \int_{X_{L}\times X_{L}}\lambda_{(y,z)}\,\textrm{d}(\nu_{L}\times\nu_{L})(y,z)
\label{eq: disintegration}
\end{equation}
be the disintegration of $\lambda$ with respect to $\pi\times\pi: X_{H}\times X_{H}\to X_{L}\times X_{L}$, where we've used the assumption that $(\pi\times\pi)_{*}\lambda = \nu_{L}\times\nu_{L}$. For $\nu_{L}\times\nu_{L}$-almost every pair $(y,z)\in X_{L}\times X_{L}, \lambda_{(y,z)}$ is a Borel probability measure on $X_{H}\times X_{H}$ supported on the fiber $\mathcal{F}_{(y,z)} = (\pi\times\pi)^{-1}\{(y,z)\} = \mathcal{F}_{y}\times\mathcal{F}_{z}$. 

Observe that $(u^{p}, u^{q})$ maps $\mathcal{F}_{(y,z)}$ to $\mathcal{F}_{(yu_{L}^{p}, zu_{L}^{q})}$. Since $\lambda$ is $(u^{p}, u^{q})$-invariant, we have for any $(y,z)\in X_{L}\times X_{L}$
\begin{equation}
(u^{p},u^{q})_{*}\lambda_{(y,z)} = \lambda_{(yu_{L}^{p}, zu_{L}^{q})}. 
\label{eq: disintegration equivariance}
\end{equation}

We showed above (Proposition \ref{prop: levi parts}) that the Levi components $(u_{L}^{p},u_{L}^{q})$ and $(u_{R}^{p}, u_{R}^{q})$ also belong to $\Lambda$, the stabilizer of $\lambda$. Since $R$ acts trivially on $X_{L}$, we therefore have
\[
\left(u_{R}^{p},u_{R}^{q}\right)_{*}\lambda_{(y,z)} = \lambda_{(y,z)}
\]
for $(\nu_{L}\times\nu_{L})$-a.e. $(y,z)$, so almost every $\lambda_{(y,z)}$ is invariant under the action of $(u_{R}^{p}, u_{R}^{q})$ on $\mathcal{F}_{(y,z)}$. It follows from (\ref{eq: fiber action}) that the image measure $\rho_{(y,z)} = (\varphi_{(y,z)}^{-1})_{*}\lambda_{(y,z)}$ on $X_{R}\times X_{R}$ is invariant for the action of $\left(\y u_{R}^{p}\y^{-1},\z u_{R}^{q}\z^{-1}\right)$. 

We now show that we may reduce to the case that almost every $\rho_{(y,z)}$ is \emph{ergodic} for this action. Naturally, we would like to demonstrate this by means of the ergodic decomposition of each of the $(u_{R}^{p},u_{R}^{q})$-invariant measures $\lambda_{(y,z)}$. The problem with this idea is that the measurable partition of $X_{H}\times X_{H}$ to which this decomposition gives rise a priori depends on $(y,z)$, creating complications in the order of logic should we attempt to average with respect to $\lambda$. However, this $(y,z)$ dependence may be removed in the following way. 

Since $\lambda$ is ergodic with respect to the action of $(u^{p},u^{q})$, it is algebraic by Ratner's measure classification theorem, and thus the probability space $(X_{H}\times X_{H},\lambda)$ may be identified with the homogeneous space $(X_{\Lambda},\nu_{\Lambda})$, where $\Lambda = \stab_{H\times H}\lambda$ as above, $X_{\Lambda} = (\Lambda\cap\Gamma\times\Gamma)\backslash\Lambda$, and $\nu_{\Lambda}$ is a $\Lambda$-invariant Borel probability measure on $X_{\Lambda}$. Since $(u_{R}^{p},u_{R}^{q})\in\Lambda$ by Proposition \ref{prop: levi parts}, we have by Proposition \ref{prop: ergodic components} that $\lambda$-a.e. ergodic component of the action of $(u_{R}^{p},u_{R}^{q})$ on $(X_{H}\times X_{H},\lambda)$ is an $N$-orbit, where $N$ is the smallest closed connected normal subgroup of $\Lambda$ such that $(u_{R}^{p},u_{R}^{q})\in N$ and $N\cap\Gamma\times\Gamma$ is a lattice in $N$. 

We clearly have $N\subset\Lambda\cap R\times R$; thus, if $(x_{1},x_{2})\in X_{H}\times X_{H}$, then the $N$-orbit $(x_{1},x_{2})N$ is contained in the fiber $\mathcal{F}_{(y_{1},y_{2})}$ where $y_{i} = \pi_{L}(x_{i})$. It follows that the measurable partition determined by the ergodic decomposition of each $\lambda_{(y,z)}$ is given by the partition of $X_{H}\times X_{H}$ into $N$-orbits, which does not depend on $(y,z)$ (note, however, that the relevant ergodic measures on the atoms of this partition do vary with $\lambda_{(y,z)}$).

To summarize, there is a measurable partition $\xi = \{C(x,x') : (x,x')\in X_{H}\times X_{H} \}$ of $X_{H}\times X_{H}$ into $(u_{R}^{p},u_{R}^{q})$-invariant subsets $C(x,x')\subset X_{H}\times X_{H}, (x,x')\in C(x,x')$, a probability measure $P$ on $X_{H}\times X_{H}/\xi$, and ergodic $(u_{R}^{p},u_{R}^{q})$-invariant Borel probability measures $\{\lambda_{(y,z)}^{C}: C\in\xi, (y,z)\in X_{L}\times X_{L}\}$ such that $\lambda_{(y,z)}^{C}$ is supported on the fiber $\mathcal{F}_{(y,z)}$ and 


\begin{align}
\int_{X_{H}\times X_{H}}f\otimes f\,\textrm{d}\lambda & = \int_{X_{L}\times X_{L}}\int_{X_{H}\times X_{H}}f\otimes f\,\textrm{d}\lambda_{(y,z)}\,\textrm{d}(\nu_{L}\times\nu_{L})(y,z) \nonumber \\
& = \int_{X_{H}\times X_{H}/\xi}\int_{X_{L}\times X_{L}}\int_{X_{H}\times X_{H}}f\otimes f\,\textrm{d}\lambda_{(y,z)}^{C}\,\textrm{d}(\nu_{L}\times\nu_{L})(y,z)\,\textrm{d}P(C).
\label{eq: decomposition}
\end{align}

Our proof of M\"{o}bius disjointness in the case that $R$ is abelain will make the crucial assumption that $f$ is in $C^{1}(X_{H})$. This causes no loss of generality at the level of our proof of Theorem \ref{thm: main}, since it should be clear (and we will show) that we can prove this theorem for arbitrary continuous functions once we have it for differentiable functions, by uniformly approximating the former class by the latter. This differentiability will enter our analysis by allowing us to control the decay of Fourier coefficients when we use harmonic analysis in the fibers of $\pi_{L}\times\pi_{L}$ (which, we recall from above, have been explicitly identified with $X_{R}\times X_{R}$). 


Observe that since the fiber measure $\lambda^{C}_{(y,z)}$ is ergodic for $(u_{R}^{p},u_{R}^{q})$ for $\nu_{L}\times\nu_{L}$-a.e. $(y,z)\in X_{L}\times X_{L}$ and $P$-a.e. $C\in X_{H}\times X_{H}/\xi$, (\ref{eq: fiber action}) implies that $\rho^{C}_{(y,z)} = (\varphi_{y}^{-1}\times\varphi_{z}^{-1})_{*}\lambda^{C}_{(y,z)}$ is ergodic for $(\y u_{R}^{p}\y^{-1}, \z u_{R}^{q}\z^{-1})$ (with the same caveats on $(y,z)$ and $C$).

Given a continuous, bounded function $f\in C_{b}(X_{H})$ and $y\in X_{L}$, we may define a continuous, bounded function $f_{y}$ on $X_{R}$ by setting $f_{y}(\Gamma_{R}r) = f(\varphi_{y}(\Gamma_{R}r)) = f(\Gamma r\y)$. We may now summarize the above discussion with the following.
\begin{prop}
Let $H$ be a connected, simply connected real Lie group, $\Gamma\subset H$ a lattice, and $u\in H$ an Ad-unipotent element, such that the $u$-orbit of the identity coset $x = \Gamma e\in X_{H}$ is uniformly distributed with respect to $\nu_{H}$. Let $H = L\rtimes R$ be the real Levi decomposition of $H$. Then the following holds for all but finitely many pairs of primes $p\neq q$: for any ergodic joining $\lambda$ of the systems $(X_{H}, u^{p}, \nu_{H})$ and $(X_{H}, u^{q}, \nu_{H})$ which is supported on the orbit closure $\overline{\{(x,x)(u^{pn},u^{qn})\}_{n\in\mathbb{N}}}$, there exist a probability space $(Z,P)$ and a family $\left\{\rho^{C}_{(y,z)} : (y,z)\in X_{L}\times X_{L}, C\in Z\right\}$ of Borel probability measures on $X_{R}\times X_{R}$ such that for any continuous, bounded function $f$ on $\Gamma\backslash H$, we have
\[
\int_{X_{H}\times X_{H}}f\otimes f\,\textrm{d}\lambda = \int_{Z}\int_{X_{L}\times X_{L}}\int_{X_{R}\times X_{R}}f_{y}\otimes f_{z} \,\textrm{d}\rho^{C}_{(y,z)}\,\textrm{d}(\nu_{L}\times \nu_{L})(y,z)\,\textrm{d}P(C).
\]
Moreover, for $(\nu_{L}\times\nu_{L})$-a.e. $(y,z)\in X_{L}\times X_{L}$ and $P$-a.e. $C\in Z, \, \rho^{C}_{(y,z)}$ is ergodic for the action of $(\y u_{R}^{p}\y^{-1},\z u_{R}^{q}\z^{-1})$ on $X_{R}\times X_{R}$.
\label{thm: reduction} 
\end{prop}
 Our interest in this theorem arises from (\ref{eq: nongeneric correlations}), which tells us that the left side of the equation in the theorem governs the behavior of correlations of $u^{p}$ and $u^{q}$, and thereby is connected to M\"{o}bius disjointness by Theorem \ref{thm: criterion}. We now show how this works in practice, beginning with the case that $R$ is abelian.

\subsection{The case of abelian $R$} We apply the formula of Proposition \ref{thm: reduction}:
\begin{equation}
\int_{X_{H}\times X_{H}}f\otimes f\,\textrm{d}\lambda = \int_{Z}\int_{X_{L}\times X_{L}}\int_{X_{R}\times X_{R}}f_{y}\otimes f_{z} \,\textrm{d}\rho^{C}_{(y,z)}\,\textrm{d}(\nu_{L}\times \nu_{L})(y,z)\,\textrm{d}P(C).
\label{eq: full}
\end{equation}
For each $y\in X_{L},\, f_{y}$ is a continuous function on the compact abelian group $X_{R}$, and therefore admits a Fourier expansion
\[
f_{y}(u) = \sum_{\chi\in\widehat{X}_{R}}\widehat{f_{y}}(\chi)\chi(u).
\]
We assume furthermore that $f\in C^{1}(X_{H})$, so that $f_{y}\in C^{1}(X_{R})$ and this series converges pointwise. This causes no loss of generality for our proof of M\"{o}bius disjointness, as we will show below.

Since $\rho^{C}_{(y,z)}$ is a probability measure, the bounded convergence theorem now allows us to exchange the order of sums and integrals for the expression inside the integral over $Z$:
\begin{align}
&\int_{X_{L}\times X_{L}}\int_{X_{R}\times X_{R}}f_{y}\otimes f_{z} \,\textrm{d}\rho^{C}_{(y,z)}\,\textrm{d}(\nu_{L}\times \nu_{L})(y,z) \nonumber\\
&\hspace{0.5cm} = \int_{X_{L}\times X_{L}}\int_{X_{R}\times X_{R}}\sum_{\chi,\psi\in\widehat{X}_{R}}\widehat{f_{y}}(\chi)\widehat{f_{z}}(\psi)\chi\otimes\psi\,\textrm{d}\rho^{C}_{(y,z)}\textrm{d}(\nu_{L}\times \nu_{L})(y,z) \nonumber\\
& \hspace{0.5cm}= \int_{X_{L}\times X_{L}}\sum_{\chi,\psi\in\widehat{X}_{R}}\widehat{f_{y}}(\chi)\widehat{f_{z}}(\psi)\int_{X_{R}\times X_{R}}\chi\otimes\psi\,\textrm{d}\rho^{C}_{(y,z)}\,\textrm{d}(\nu_{L}\times \nu_{L})(y,z)
\label{eq: bounded convergence}
\end{align}
We are not yet justified to move the sum over $\chi,\psi$ further past the integral over $X_{L}\times X_{L}$. In order to do this, observe that since $R$ is simply connected and abelian, we may identify $\widehat{X_{R}}$ with $\mathbb{Z}^{k}$, where $k = \dim{R}$. Let
\[
\mathbb{Z}^{k}_{N} = \{m\in\mathbb{Z}^{k} : |m|\leq N\},
\]
where $|(m_{1},\dots,m_{k})| = |m_{1}|+\cdots+|m_{k}|$. Identifying each character $\chi\in\widehat{X_{R}}$ with the corresponding $m\in\mathbb{Z}^{k}$, we may unambiguously consider $\widehat{f_{y}}(m)$, and we claim that the function on $X_{L}\times X_{L}$ given by
\[
(y,z)\mapsto \sum_{\chi,\psi\in\widehat{X}_{R}}\widehat{f_{y}}(\chi)\widehat{f_{z}}(\psi)\int_{X_{R}\times X_{R}}\chi\otimes\psi\,\textrm{d}\rho^{C}_{(y,z)},
\] 
which can equivalently be written as
\[
(y,z)\mapsto \sum_{m,m'\in\mathbb{Z}^{k}}\widehat{f_{y}}(m)\widehat{f_{z}}(m')\int_{X_{R}\times X_{R}}\chi_{m}\otimes\chi_{m'}\,\textrm{d}\rho^{C}_{(y,z)},
\]
is the pointwise limit of the sequence of functions defined for $N\geq 1$ by
\[
(y,z)\mapsto \sum_{m,m'\in\mathbb{Z}^{k}_{N}}\widehat{f_{y}}(m)\widehat{f_{z}}(m')\int_{X_{R}\times X_{R}}\chi_{m}\otimes\chi_{m'}\,\textrm{d}\rho^{C}_{(y,z)}.
\]
Indeed, if $y,z$ are fixed, then the difference between the full sum over $m,m'$ and its $N$-truncated version is
\[
\sum_{\substack{m,m'\in\mathbb{Z}^{k} \\ |m|,|m'| >N}}\widehat{f_{y}}(m)\widehat{f_{z}}(m')\int_{X_{R}\times X_{R}}\chi_{m}\otimes\chi_{m'}\,\textrm{d}\rho^{C}_{(y,z)}.
\]
We claim that this sum converges absolutely, and that its absolute value goes to 0 as $N\to\infty$. Notice that $|\int_{X_{R}\times X_{R}}\chi_{m}\otimes\chi_{m'}\,\textrm{d}\rho^{C}_{(y,z)}| \leq 1$, because $\chi_{m},\chi_{m'}$ are characters and $\rho^{C}_{(y,z)}$ is a probability measure. Thus, the absolute value of this expression is bounded by
\[
 \sum_{\substack{m,m'\in\mathbb{Z}^{k} \\ |m|,|m'| >N}}\left|\widehat{f_{y}}(m)\widehat{f_{z}}(m')\right|,
\]
assuming this expression converges. 

We have already assumed above that $f\in C^{1}(X_{H})$; we now go further and assume that $f\in C^{2}(X_{H})$, which will be justified in a similar fashion below. We then have the standard decay of Fourier coefficients:
\[
\left|\widehat{f_{y}}(m)\right|\leq\frac{M(f_{y})}{(1+ |m|)^{2}},
\]
where $M(f_{y})>0$ depends only on $f_{y}$, and similarly for $\widehat{f_{z}}(m')$. Therefore, the expression of interest is bounded by
\[
M(f_{y})M(f_{z})\sum_{\substack{m,m'\in\mathbb{Z}^{k} \\ |m|,|m'| >N}}\frac{1}{(1+|m|)^{2}}\cdot\frac{1}{(1+|m'|)^{2}}.
\]
This clearly converges to 0 as $N\to\infty$, because the corresponding sum over all $m,m'\in\mathbb{Z}^{k}$ converges.

This yields the desired pointwise convergence, and thus we may again apply the bounded convergence theorem to conclude in all that
\begin{align}
& \int_{X_{L}\times X_{L}}\sum_{\chi,\psi\in\widehat{X}_{R}}\widehat{f_{y}}(\chi)\widehat{f_{z}}(\psi)\int_{X_{R}\times X_{R}}\chi\otimes\psi\,\textrm{d}\rho^{C}_{(y,z)}\,\textrm{d}(\nu_{L}\times \nu_{L})(y,z) \nonumber\\
& \hspace{0.5cm} = \sum_{\chi,\psi\in\widehat{X}_{R}}\int_{X_{L}\times X_{L}}\widehat{f_{y}}(\chi)\widehat{f_{z}}(\psi)\int_{X_{R}\times X_{R}}\chi\otimes\psi\,\textrm{d}\rho^{C}_{(y,z)}\,\textrm{d}(\nu_{L}\times \nu_{L})(y,z).
\label{eq: integrate characters}
\end{align}

By Proposition \ref{thm: reduction}, $\rho^{C}_{(y,z)}$ is ergodic for a unipotent translation on $X_{R}\times X_{R}$, and therefore is algebraic by Ratner's measure classification theorem, meaning it is supported on an orbit of $\stab_{(R\times R)}\,\rho^{C}_{(y,z)} := S^{C}_{(y,z)}$. Moreover, the point $(r_{C},s_{C})$ on whose $S^{C}_{(y,z)}$ orbit the measure $\rho^{C}_{(y,z)}$ is supported does not depend on $(y,z)$. This is because, following our discussion above, the ergodic decomposition of almost every $\lambda_{(y,z)}$ for the action of $(u_{R}^{p},u_{R}^{q})$ on $X_{H}\times X_{H}$ equals the partition into $N$-orbits, where $N$ is determined solely by $\lambda$. Thus, the points $(r_{C},s_{C})$ parametrize these $N$-orbits in the fibers of $\pi_{L}\times\pi_{L}$, which we have identified with $X_{R}\times X_{R}$. As we have $N\subset R\times R$, the points $(r_{C},s_{C})\in X_{R}\times X_{R}$ are determined by the partition of $X_{R}\times X_{R}$ into $N$-orbits (each $C$ is an $N$-orbit, and we have $(r_{C},s_{C})\in C$), with no dependence on $(y,z)$.  

 
Consequently, we have
\[
\int_{X_{R}\times X_{R}}\chi(r)\psi(s)\,\textrm{d}\rho^{C}_{(y,z)} = \begin{cases}
\chi(r_{C})\psi(s_{C}) & \text{ if } (\chi\otimes\psi)|_{S^{C}_{(y,z)}} \text{ is trivial} \\
0 & \text{ else}. 
\end{cases}
\]

Since $S^{C}_{(y,z)}$ contains $(\y u_{R}^{p}\y^{-1},\z u_{R}^{q}\z^{-1})$, and since $\rho^{C}_{(y,z)}$ is ergodic for the action thereof, we see that 
\[
\int_{X_{R}\times X_{R}}\chi\otimes\psi\,\textrm{d}\rho^{C}_{(y,z)} = \begin{cases}
\chi(r_{C})\psi(s_{C}) & \text{ if } \chi(\y u_{R}^{p}\y^{-1})\psi(\z u_{R}^{q}\z^{-1}) = 1 \\
0 & \text { if not}.
\end{cases}
\]
Therefore, (\ref{eq: full}) and (\ref{eq: integrate characters}) together yield the following formula:
\begin{equation}
\int_{X_{H}\times X_{H}}f\otimes f\,\textrm{d}\lambda = \int_{Z}\sum_{\chi,\psi\in\widehat{X}_{R}}\chi(r_{C})\psi(s_{C})\int_{X_{L}^{p,q}(\psi,\chi)}\widehat{f}_{y}(\chi)\widehat{f}_{z}(\psi)\,\textrm{d}(\nu_{L}\times\nu_{L})(y,z)\,\textrm{d}P(C)
\label{eq: equality}
\end{equation}
where
\[
X_{L}^{p,q}(\chi,\psi) = \left\{(y,z)\in X_{L}\times X_{L} : \chi(\y u_{R}^{p}\y^{-1})\psi(\z u_{R}^{q}\z^{-1})=1\right\}. 
\]

Now, introduce the set
\[
L^{p,q}(\psi,\chi) = \left\{(\ell_{1},\ell_{2})\in L\times L : \chi(\ell_{1}u_{R}^{p}\ell_{1}^{-1})\psi(\ell_{2}u_{R}^{q}\ell_{2}^{-1})=1\right\}.
\]
Observe that $L^{p,q}(\psi,\chi)$ is an analytic subvariety of $L\times L$, since the group operations in $H$ and the characters $\psi,\chi$ on $R$ are analytic. However, any analytic subvariety of $L\times L$ which is not equal to $L\times L$ has Haar measure zero (it is well-known that, in any system of local coordinates, the Haar measure is absolutely continuous with respect to the Lebesgue measure). Since $X_{L}^{p,q}(\psi,\chi)\subseteq (\pi\times\pi)(L^{p,q}(\psi,\chi))$, where $\pi: L\to X_{L}$ is the projection, we see that we have the following dichotomy: either
\begin{itemize}
\item $L^{p,q}(\psi,\chi) = L\times L$, so $X_{L}^{p,q}(\psi,\chi) = X_{L}\times X_{L}$ and $(\nu_{L}\times\nu_{L})(X_{L}^{p,q}(\psi,\chi)) = 1$, or\\
  
\item $L^{p,q}(\psi,\chi) \neq L\times L$, and then $(\nu_{L}\times \nu_{L})(X_{L}^{p,q}(\psi,\chi)) = 0$.
\end{itemize}
Consequently, (\ref{eq: equality}) becomes
\begin{equation}
\int_{X_{H}\times X_{H}}f\otimes f\,\textrm{d}\lambda = \int_{Z}\sum_{\substack{\chi,\psi \\ L^{p,q}(\psi,\chi) = L\times L}}\widehat{f}(\chi)\widehat{f}(\psi)\chi(r_{C})\psi(s_{C})\,\textrm{d}P(C)
\label{eq: almost there}
\end{equation}
where
\[
\widehat{f}(\chi) = \int_{X_{L}}\widehat{f}_{y}(\chi)\,\textrm{d}y.
\]

Observe, however, that if $L^{p,q}(\psi,\chi) = L\times L$ then we have in particular $\chi(u_{R}^{p})\psi(u_{R}^{q}) = 1$, which implies by the ergodicity of $u_{R}$ acting on $X_{R}$ that $\chi^{p} = \psi^{-q}$. As this will turn out to be the central point of our argument below, we write (\ref{eq: almost there}) as
\begin{equation}
\int_{X_{H}\times X_{H}}f\otimes f\,\textrm{d}\lambda = \int_{Z}\sideset{}{'}\sum_{\substack{\chi,\psi \\ \chi^{p} = \psi^{-q}}}\widehat{f}(\chi)\widehat{f}(\psi)\chi(r_{C})\psi(s_{C})\,\textrm{d}P(C),
\label{eq: calculate joinings}
\end{equation}
where the primed summation is restricted to those pairs of characters such that $L^{p,q}(\psi,\chi) = L\times L$.

We are now in a position to prove M\"{o}bius disjointness in the abelian case, yielding the base of induction on the step of the nilradical of $H$.



\begin{proof}[Proof of Theorem \ref{thm: main} for the case of abelian $R$.] \label{abelian proof} First, we may suppose without loss of generality that $\int_{X_{H}}f\,\textrm{d}\nu_{H} = 0$. Indeed, for any constant $c$ we have
\[
\sum_{n\leq N}\mu(n)f(xu^{n}) = \sum_{n\leq N}\mu(n)(f(xu^{n}) - c) + c\sum_{n\leq N}\mu(n) = \sum_{n\leq N}\mu(n)(f(xu^{n}) - c) + o(N)
\]
by the prime number theorem. Since $x\mapsto f(x) - c$ remains a continuous, bounded function on $X_{H}$, all of our work above applies to this function in place of $f$, with $c = \int_{X_{H}}f\,\textrm{d}\nu_{H}$. 

Furthermore, we may assume that $f$ belongs to $C^{1}(X_{H})$. Indeed, if $f$ is an arbitrary continuous function, and $f_{k}$ a sequence of $C^{1}$ functions which converge to $f$ uniformly, then for any $\epsilon > 0$ we have for $N\geq 1$ and all sufficiently large $k$
\[
\frac{1}{N}\sum_{n\leq N}\mu(n)f(xu^{n}) = \frac{1}{N}\sum_{n\leq N}\mu(n)f_{k}(xu^{n}) + O(\epsilon).
\]
If M\"{o}bius disjointness is known for $C^{1}$ functions, we may choose $k$ so large that the first term above is smaller than $\epsilon$ as well, making it clear that limit of the expression on the left as $N\to\infty$ must be zero.

Now, notice that 
\begin{align*}
\widehat{f}(\chi_{0}) & = \int_{X_{L}}\widehat{f}_{y}(\chi_{0})\,\textrm{d}y \\
& =  \int_{X_{L}}\int_{X_{R}}f(\Gamma r\y)\,\textrm{d}r\,\textrm{d}y \\
& = \int_{X_{H}}f\,\textrm{d}\nu_{H} \\
& = 0.
\end{align*}
Therefore, the expression on the right side of (\ref{eq: calculate joinings}) becomes
\[
 \int_{Z}\sideset{}{'}\sum_{\substack{\chi,\psi \\ \chi^{p} = \psi^{-q}}}\widehat{f}(\chi)\widehat{f}(\psi)\chi(r_{C})\psi(s_{C})\,\textrm{d}P(C) =  \int_{Z}\sideset{}{'}\sum_{\substack{\chi,\psi\neq\chi_{0} \\ \chi^{p} = \psi^{-q}}}\widehat{f}(\chi)\widehat{f}(\psi)\chi(r_{C})\psi(s_{C})\,\textrm{d}P(C).
\]
Since $R$ is simply connected and abelian, it is isomorphic to $\mathbb{R}^{k}$, where $k = \dim{R}$. Therefore $X_{R}$ is isomorphic to the torus $\mathbb{T}^{k}$, and $\widehat{X_{R}}$ may be identified with $\mathbb{Z}^{k}$, from which it is easily seen that
\[
\int_{Z} \sideset{}{'}\sum_{\substack{\chi,\psi\neq\chi_{0} \\ \chi^{p} = \psi^{-q}}}\widehat{f}(\chi)\widehat{f}(\psi)\chi(r_{C})\psi(s_{C})\,\textrm{d}P(C) = \int_{Z}\sideset{}{'}\sum_{\chi\neq\chi_{0}}\widehat{f}(\chi^{-q})\widehat{f}(\chi^{p})\chi^{-q}(r_{C})\chi^{p}(s_{C})\,\textrm{d}P(C).
\]
We now write out the integral over $X_{L}$ in the definition of $\widehat{f}(\chi)$ in order to estimate these coefficients by averaging over the restrictions of $f$ to fibers, which gives in all
\begin{equation}
\int_{X_{H}\times X_{H}}f\otimes f\,\textrm{d}\lambda = \int_{Z}\int_{X_{L}\times X_{L}} \sideset{}{'}\sum_{\chi\neq\chi_{0}}\chi^{-q}(r_{C})\chi^{p}(s_{C})\widehat{f}_{y}(\chi^{-q})\widehat{f}_{z}(\chi^{p})\,\textrm{d}y\,\textrm{d}z\,\textrm{d}P(C).
\label{eq: exchange}
\end{equation}
%

For each $y\in X_{L}$, the functions $\chi\mapsto \widehat{f}_{y}(\chi^{-q})$ and $\chi\mapsto\widehat{f}_{y}(\chi^{p})$ belong to $\ell^{2}(\widehat{X}_{R})$ by Plancherel's theorem. The same is then true of the functions $\chi\mapsto \widehat{f}_{y}(\chi^{-q})\chi^{-q}(r_{C})$ and $\chi\mapsto\widehat{f}_{y}(\chi^{p})\chi^{p}(s_{C})$ since $|\chi| \equiv 1$. Thus, for each $y,z\in X_{L}$ and $C\in Z$, we have by the Cauchy-Schwarz inequality
\[
\left|\sideset{}{'}\sum_{\chi\neq\chi_{0}}\chi^{-q}(r_{C})\chi^{p}(s_{C})\widehat{f}_{y}(\chi^{-q})\widehat{f}_{z}(\chi^{p})\right|^{2}\leq \sideset{}{'}\sum_{\chi\neq\chi_{0}}\left|\widehat{f}_{y}(\chi^{-q})\right|^{2}\sideset{}{'}\sum_{\chi\neq\chi_{0}}\left|\widehat{f}_{z}(\chi^{p})\right|^{2},
\]
and the sums on the right side both converge. Note that this estimate does not depend on $C$. Hence, on combining this with (\ref{eq: exchange}) and using the triangle inequality to eliminate the integral over $Z$, we get the estimate 
\begin{equation}
\left|\int_{X_{H}\times X_{H}}f\otimes f\,\textrm{d}\lambda\right|\leq \int_{X_{L}\times X_{L}}\left(\sum_{\chi\neq\chi_{0}}\left|\widehat{f}_{y}(\chi^{-q})\right|^{2}\sum_{\chi\neq\chi_{0}}\left|\widehat{f}_{z}(\chi^{p})\right|^{2}\right)^{1/2}\,\textrm{d}y\,\textrm{d}z,
\label{eq: joining inequality}
\end{equation}
where the positivity allows us to remove the prime on the summation.

We now use the standard decay of Fourier coefficients for $C^{1}$ functions. According to this, since each $f_{y}$ is differentiable on $X_{R}$, we have for each $m = (m_{1},\dots,m_{k})\in \mathbb{Z}^{k}$ (now identifying $\widehat{X_{R}}$ with $\mathbb{Z}^{k}$, so that $\widehat{f}_{y}(m)$ makes sense)
\[
|\widehat{f}_{y}(m)| \leq c_{k}\frac{\max(\|f_{y}\|_{L^{1}},\max_{i=1,\dots, k}\|\frac{\partial}{\partial x_{i}}f_{y}\|_{L^{1}})}{1 + |m|},
\]
where $c_{k} > 0$ depends only on $k$, and $(x_{1},\dots,x_{k})$ are the coordinates on $\mathbb{T}^{k}$. Let $M(f_{y}) = \max(\|f_{y}\|_{L^{1}},\max_{i=1,\dots, k}\|\frac{\partial}{\partial x_{i}}f_{y}\|_{L^{1}})$. Then we see
\begin{align*}
\sum_{m\neq 0}\left|\widehat{f}_{y}(qm)\right|^{2} & \leq c_{k}^{2}M(f_{y})^{2}\sum_{m\neq 0}\frac{1}{(1 + |qm|)^{2}} \\
& = \frac{c_{k}^{2}M(f_{y})^{2}}{q^{2}}\sum_{m\neq 0}\frac{1}{(q^{-1} + |m|)^{2}} \\
& \leq C_{k}M(f_{y})^{2}{q^{-2}},
\end{align*}
where $C_{k}$ is $c_{k}^{2}$ multiplied by the value of the convergent sum $\sum_{m\neq 0\in \mathbb{Z}^{k}}|m|^{-2}$. 

Plugging this inequality into (\ref{eq: joining inequality}), we find
\begin{equation}
\left|\int_{X_{H}\times X_{H}}f\otimes f\,\textrm{d}\lambda\right| \leq \frac{C_{k}^{2}}{pq}\int_{X_{L}\times X_{L}}M(f_{y})M(f_{z})\,\textrm{d}y\,\textrm{d}z.
\label{eq: almost}
\end{equation}
Our only remaining obstacle is to show that the integral on the right side of this expression converges. Observe that 
\[
\|f_{y}\|_{L^{1}} = \int_{X_{R}}|f(\Gamma r\y)|\,\textrm{d}\nu_{R}\leq \|f\|_{\infty},
\]
so $M(f_{y})\leq \max(\|f\|_{\infty},\max_{i=1,\dots, k}\|\frac{\partial}{\partial x_{i}}f_{y}\|_{L^{1}(X_{R})})$. At the same time, since $\pi: X_{H}\to X_{L}$ is a surjective submersion, we may introduce local coordinates on $X_{H}$ around $x$ in such a way that the coordinates of $x$ are $(x_{1},\dots,x_{k},y_{1},\dots,y_{s})$, where $(x_{1},\dots,x_{k})$ are coordinates on $\mathbb{T}^{k}$ and $(y_{1},\dots, y_{s})$ are the coordinates of $y = \pi(x)\in X_{L}$. In this system of local coordinates, we have
\[
\frac{\partial}{\partial x_{i}}f_{y}(\Gamma_{R}r) = \frac{\partial}{\partial x_{i}}f(x),
\]
where $x = \Gamma r\y$. Since $f$ is assumed to be differentiable, we have $(\partial/\partial x_{i})f\in L^{1}(X_{H})$; thus, we find
\begin{align*}
\int_{X_{L}}\left\|\frac{\partial}{\partial x_{i}}f_{y}\right\|_{L^{1}(X_{R})}\,\textrm{d}y & = \int_{X_{L}}\int_{X_{R}}\left|\frac{\partial}{\partial x_{i}}f_{y}(\Gamma_{R}r)\right|\,\textrm{d}\nu_{R}\,\textrm{d}y  \\
& = \int_{X_{H}}\left|\frac{\partial}{\partial x_{i}}f(x)\right|\,\textrm{d}x \\
& = \left\|\frac{\partial}{\partial x_{i}}f\right\|_{L^{1}(X_{H})}.
\end{align*} 
Consequently, we have
\begin{align*}
\int_{X_{L}}M(f_{y})\,\textrm{d}y & \leq \max\left(\|f\|_{\infty},\max_{i=1,\dots, k}\int_{X_{L}}\left\|\frac{\partial}{\partial x_{i}}f_{y}\right\|_{L^{1}(X_{R})}\,\textrm{d}y\right) \\
& =  \max\left(\|f\|_{\infty},\max_{i=1,\dots, k}\left\|\frac{\partial}{\partial x_{i}}f\right\|_{L^{1}(X_{H})}\right) <\infty.
\end{align*}
Thus, (\ref{eq: almost}) finally yields
\begin{equation}
\left|\int_{X_{H}\times X_{H}}f\otimes f\,\textrm{d}\lambda\right| \leq \frac{1}{pq}C_{f},
\label{eq: final}
\end{equation}
where $C_{f}>0$ only depends on $f$, not on $p$ or $q$. Putting things together, we see that if $\tau > 0$ and $p,q > (\tau C_{f})^{-1/2}$, then
\[
\left|\int_{X_{H}\times X_{H}}f\otimes f\,\textrm{d}\lambda\right| < \tau. 
\]
Since $(\tau C_{f})^{-1/2} = o(e^{1/\tau})$, the proof of Theorem \ref{thm: criterion} \cite{BSZ} therefore implies that 
\[
\sum_{n\leq N}\mu(n)f(xu^{n}) = o(N),
\]
as desired.

\end{proof}

\subsection{The case of general $R$} We now prove Theorem \ref{thm: main} when $R$, the radical of $H$, is not necessarily abelian (recall that $R$ is nilpotent by Proposition \ref{prop: nilpotent}). Our argument proceeds by induction on $d$, the step of the nilpotent group $R$, with the base case $d=1$ carried out above. In order to mimic the proof in the abelian case, we will use harmonic analysis on the nilpotent group $R$, and specifically the notions of ``vertical Fourier analysis" (\cite{GT},\cite{Zo}). Let $R$ be a nilpotent group of nilpotency class $d$ (i.e. its descending central series has length $d+1$), and $\Gamma\subset R$ a lattice.  A continuous function $g$ on the nilmanifold $\Gamma\backslash R$ is called a vertical character if there exists a character $\chi$ of the compact abelian group $\Gamma\cap R_{d}\backslash R_{d} = \Gamma_{d}\backslash R_{d}$ such that 
\[
g(vr_{d}) = g(v)\chi(\Gamma_{d}r_{d}) \text{ for all } v\in \Gamma\backslash R, r_{d}\in R_{d}.
\]
Given any continuous function $g$ on $\Gamma\backslash R$ and character $\chi\in \widehat{\Gamma_{d}\backslash R_{d}}$, we define 
\[
g^{\chi}(v) := \int_{\Gamma_{d}\backslash R_{d}}g(vr_{d})\overline{\chi(r_{d})}\,\textrm{d}r_{d}. 
\]
Then $g^{\chi}$ is a vertical character of $\Gamma\backslash R$, with $\chi$ the associated character of $\Gamma_{d}\backslash R_{d}$, and we have the vertical Fourier series 
\[
g = \sum_{\chi\in\widehat{\Gamma_{d}\backslash R_{d}}}g^{\chi} \text{ in } L^{2}(\Gamma\backslash R),
\]
with pointwise convergence if e.g. $g$ is differentiable. 

We now apply this expansion in the setting of Theorem \ref{thm: criterion}, assuming $f\in C^{1}(X_{H})$. By (\ref{eq: nongeneric correlations}), we have for any $p\neq q$
\begin{equation}
\lim_{N\to\infty}\frac{1}{N}\sum_{n\leq N}f(xu^{pn})\overline{f(xu^{qn})} = \int_{X_{H}\times X_{H}}f\otimes \overline{f} \,\textrm{d}\lambda
\label{eq: reiterate}
\end{equation}
for an appropriate joining $\lambda$. We prove by induction on $d$ that there is a constant $C_{f} >0 $ depending only on $f$ such that
\[
\left|\int_{X_{H}\times X_{H}}f\otimes \overline{f} \,\textrm{d}\lambda\right|\leq\frac{1}{pq}C_{f}.
\]
By Proposition \ref{thm: reduction}, we have for all but finitely many $p\neq q$ that (using the notation therein)
\begin{align}
\int_{X_{H}\times X_{H}}f\otimes f\,\textrm{d}\lambda & = \int_{Z}\int_{X_{L}\times X_{L}}\int_{X_{R}\times X_{R}}f_{y}\otimes f_{z} \,\textrm{d}\rho^{C}_{(y,z)}\,\textrm{d}(\nu_{L}\times \nu_{L})(y,z)\,\textrm{d}P(C) \nonumber \\
& = \sum_{\chi,\psi\in\widehat{\Gamma_{d}\backslash R_{d}}}\int_{Z}\int_{X_{L}\times X_{L}}\int_{X_{R}\times X_{R}}f_{y}^{\chi}\otimes\overline{f_{z}^{\psi}}\,\textrm{d}\rho^{C}_{(y,z)}\,\textrm{d}(\nu_{L}\times\nu_{L})(y,z)\,\textrm{d}P(C),
\label{eq: expand}
\end{align}
where the exchange of sums and integrals is justified by the pointwise convergence of the vertical Fourier series and the bounded convergence theorem, since each $f_{y}$ is differentiable on $X_{R}$.

Now, the triviality of $\chi_{0}$ means that for each $y\in X_{L}$, the vertical character $f_{y}^{\chi_{0}}$ on $X_{R}$ is $R_{d}$-invariant; thus it passes to a continuous function on the nilmanifold $\Gamma\backslash R/ R_{d}$. The nilpotent group $R/R_{d}$ has strictly lower step than $R$, and thus we may apply the inductive hypothesis to assert that
\begin{equation}
\left|\int_{Z}\int_{X_{L}\times X_{L}}\int_{X_{R}\times X_{R}}f_{y}^{\chi_{0}}\otimes\overline{f_{z}^{\chi_{0}}}\,\textrm{d}\rho^{C}_{(y,z)}\,\textrm{d}(\nu_{L}\times\nu_{L})(y,z)\,\textrm{d}P(C)\right| \leq \frac{1}{pq}C_{f},
\label{eq: trivial}
\end{equation}
where $C_{f}$ depends only on $f$. 

It remains to deal with
\[
 \sideset{}{'}\sum_{\chi,\psi}\int_{Z}\int_{X_{L}\times X_{L}}\int_{X_{R}\times X_{R}}f_{y}^{\chi}\otimes \overline{f}_{z}^{\psi} \,\textrm{d}\rho^{C}_{(y,z)}\,\textrm{d}(\nu_{L}\times \nu_{L})(y,z)\,\textrm{d}P(C),
\]
where the primed summation runs over pairs of characters $\chi,\psi$, at least one of which is nontrivial.

Following \cite{Z}, we claim that for $\nu_{L}\times\nu_{L}$-a.e. $(y,z)\in X_{L}\times X_{L}$ and $P$-a.e. $C\in Z$, the support of $\rho^{C}_{(y,z)}$ is invariant under the action of $R_{d}\times R_{d}$. Indeed, by a theorem of Leon Green \cite{AGH}, if $G$ is a nilpotent group and $\Gamma\subset G$ a lattice, then a translation on $\Gamma\backslash G$ is ergodic relative to the Haar measure if and only if its projection to the horizontal torus $\Gamma\backslash G/[G,G]$ is ergodic. Letting $X_{R}^{\text{hor}}$ be the horizontal torus associated to $X_{R}$, it follows from the ergodicity of $\rho_{(y,z)}^{C}$ for $(\y u_{R}^{p}\y^{-1},\z u_{R}^{q}\z^{-1})$ that for $\nu_{L}\times\nu_{L}$-a.e. $(y,z)\in X_{L}\times X_{L}$ and $P$-a.e. $C\in Z$ we have
\[
\stab_{R\times R}\left(\rho_{(y,z)}^{C}\right) = \bigcap_{\substack{\omega,\xi\in \widehat{X_{R}^{\text{hor}}} \\ \omega(\y u_{R}^{p}\y^{-1})\xi(\z u_{R}^{q}\z^{-1}) = 1}}\ker(\omega\otimes\xi)
\]
(where we use the same symbol to denote a character of $X_{R}^{\text{hor}}$ and its pullback to $R$). 

Since $d >  1$, we have $R_{d}\subset [R,R]$. It follows that $R_{d}\times R_{d}\subset\ker(\omega\otimes\xi)$ for every $\omega,\xi\in \widehat{X_{R}^{\text{hor}}}$, which proves the desired $R_{d}\times R_{d}$-invariance.

Hence, for any $r_{d},s_{d}\in R_{d}\times R_{d}$ and any $\chi,\psi\in \widehat{\Gamma_{d}\backslash R_{d}}$, at least one of which is nontrivial, we have
\begin{align*}
\int_{X_{R}\times X_{R}}f_{y}^{\chi}\otimes \overline{f}_{z}^{\psi} \,\textrm{d}\rho^{C}_{(y,z)} & = \int_{X_{R}\times X_{R}}f_{y}^{\chi}(vr_{d})\overline{f}_{z}^{\psi}(ws_{d}) \,\textrm{d}\rho^{C}_{(y,z)}(v,w) \\
& = \chi(\Gamma_{d}r_{d})\psi(\Gamma_{d}s_{d})\int_{X_{R}\times X_{R}}f_{y}^{\chi}\otimes \overline{f}_{z}^{\psi} \,\textrm{d}\rho^{C}_{(y,z)},
\end{align*}
and it follows that
\[
\int_{X_{R}\times X_{R}}f_{y}^{\chi}\otimes \overline{f}_{z}^{\psi} \,\textrm{d}\rho^{C}_{(y,z)} = 0.
\]
Combining this with (\ref{eq: expand}), we see that 
\[
\int_{X_{H}\times X_{H}}f\otimes \overline{f}\,\textrm{d}\lambda = \int_{Z}\int_{X_{L}\times X_{L}}\int_{X_{R}\times X_{R}}f_{y}^{\chi_{0}}\otimes\overline{f_{z}^{\chi_{0}}}\,\textrm{d}\rho^{C}_{(y,z)}\,\textrm{d}(\nu_{L}\times\nu_{L})(y,z)\,\textrm{d}P(C),
\]
and thus, by (\ref{eq: trivial}),
\[
\left|\int_{X_{H}\times X_{H}}f\otimes \overline{f}\,\textrm{d}\lambda\right|\leq\frac{1}{pq}C_{f}.
\]
Theorem \ref{thm: criterion} now yields
\[
\sum_{n\leq N}\mu(n)f(xu^{n}) = o(N),
\]
proving Theorem \ref{thm: main}. \qed

\end{document}